\documentclass[11pt,reqno]{amsart}
\usepackage{amssymb,mathtools,calc,verbatim,enumitem,tikz,url,mathrsfs,cite,fullpage}
\usepackage{bbm}
\usepackage{textcomp}
\usepackage{setspace}
\usepackage{amsthm}
\usepackage{amsmath}
\usepackage{float}
\usepackage{graphicx}
\usepackage{marvosym}
\usepackage{empheq}
\usepackage{latexsym}
\usepackage[T1]{fontenc}
\usepackage{color}
\usepackage{hyperref}
\hypersetup{pdftitle={Exponential tails for factors and the chromatic number of random graphs},pdfauthor={Zhifei Yan}}
\usepackage{dsfont}
\usepackage{microtype}

\newenvironment{poc}{\begin{proof}[Proof of claim]}{\end{proof}}
\newtheorem{theorem}{Theorem}[section]
\newtheorem{lemma}[theorem]{Lemma}
\newtheorem{corollary}[theorem]{Corollary}

\newtheorem{proposition}[theorem]{Proposition}
\newtheorem{claim}[theorem]{Claim}

\theoremstyle{definition}
\newtheorem{definition}[theorem]{Definition}

\newtheorem*{qu*}{Question}
\theoremstyle{remark}

\usepackage{cleveref}

\newcommand\E{\operatorname{\mathbb{E}}}
\newcommand\cA{\mathcal{A}}
\newcommand\cB{\mathcal{B}}
\newcommand\cC{\mathcal{C}}
\newcommand\cD{\mathcal{D}}

\newcommand\cH{\mathcal{H}}
\newcommand\cI{\mathcal{I}}

\newcommand\cM{\mathcal{M}}

\renewcommand\Pr{\operatorname{\mathbb{P}}}

\renewcommand\leq{\leqslant}

\renewcommand\le{\leqslant}
\renewcommand\ge{\geqslant}
\renewcommand\to{\rightarrow}

\def\EE{\mathbb{E}}

\def\Var{\mathrm{Var}}

\newcommand{\ind}{\mathbbm{1}}
\newcommand{\cO}{\mathcal{O}}
\newcommand{\cR}{\mathcal{R}}
\newcommand{\cU}{\mathcal{U}}
\newcommand{\cV}{\mathcal{V}}
\newcommand{\cZ}{\mathcal{Z}}
\newcommand{\med}{\operatorname{med}}
\allowdisplaybreaks

\title{Exponential tails for factors and\\
the chromatic number of random graphs}
\author{Zhifei Yan}
\address{ECOPRO, Institute for Basic Science, 55 Expo-ro, Yuseong-gu,
Daejeon, 34126, Korea}
\email{zhifeiyan@ibs.re.kr}

\makeatletter
\let\manuscript@tocwrite\@tocwrite
\renewcommand{\@tocwrite}[2]{%
  \ifnum\@toclevel=\z@
    \ifx\@secnumber\@empty
    \else
      \manuscript@tocwrite{#1}{#2}%
    \fi
  \else
    \manuscript@tocwrite{#1}{#2}%
  \fi
}
\makeatother

\begin{document}

\thanks{Z.Y. was supported by the Institute for Basic Science (IBS-R029-C4).}

\begin{abstract}
The celebrated result of Johansson, Kahn and Vu determined the threshold
order for clique factors in random graphs, and subsequent work identified
the sharp threshold and the corresponding hitting-time phenomenon.  In
this paper we study the probability that there is no $K_r$-factor above
the threshold and, more generally, the probability that the largest
$K_r$-matching covers less than $n-s$ vertices of $G(n,p)$.  For every fixed
$r\ge3$ and every non-negative integer $s=s(n)$, throughout the range
$$
 n^{-2/r}(\log n)^{1/\binom r2}\ll p\ll n^{-2/(r+1)},
 \qquad n-s\in r\mathbb Z,
 \qquad s=o(n),
$$
we prove
$$
 \mathbb P\bigl(\phi_r^s(G(n,p))=0\bigr)
 =\exp\!\left(-\Theta_r\!\left((s+1)\frac{\mu_r(n,p)}n\right)\right),
$$
where $\phi_r^s(G)$ is the number of $K_r$-matchings covering exactly
$n-s$ vertices and
$\mu_r(n,p):=\binom nrp^{\binom r2}$.  The lower bound is given by
$s+1$ vertices which lie in no copy of $K_r$.  For the upper bound we
develop an iterable one-root version of the Johansson--Kahn--Vu method.

As a structural consequence, we show that the remainder of $G(n,p)$
outside every maximal $K_r$-matching has an almost-perfect
$K_{r-1}$-matching throughout the sparse clique window.  Independently,
we prove a central limit theorem for the maximum $K_r$-matching number.
Combining these inputs and a structural theorem for $r=2$ from our earlier work, we prove a central limit theorem for the chromatic number of very dense random graphs: for every $r\ge2$ and
 $n^{-2/r}(\log n)^{1/\binom r2}\ll p\ll n^{-2/(r+1)},$
$$
 \frac{\chi(G(n,1-p))-\E\chi(G(n,1-p))}
      {\sqrt{\mu_{r+1}(n,p)}/r}
 \xrightarrow{\mathrm d}\mathcal N(0,1),
 \qquad
 \Var\bigl(\chi(G(n,1-p))\bigr)
 \sim\frac{\mu_{r+1}(n,p)}{r^2}.
$$
This settles the Surya--Warnke conjecture throughout the interior of every clique window with $r\ge2$, strengthening its concentration prediction to a Gaussian limit with asymptotically exact variance.
\end{abstract}
\maketitle


\section{Introduction}
Perfect and almost-perfect packings are among the central spanning
structures in random graphs.  Given a fixed graph $F$, an $F$-factor is
a collection of vertex-disjoint copies of $F$ covering the entire
vertex set.  Johansson, Kahn and Vu~\cite{JKV} determined the threshold
order for the appearance of $F$-factors when $F$ is strictly
$1$-balanced.  In particular, the threshold order for a $K_r$-factor is
\[
 n^{-2/r}(\log n)^{1/\binom r2}.
\]
The leading constant was obtained by combining Kahn's asymptotically
sharp solution of Shamir's hypergraph matching problem~\cite{Kahn23}
with coupling results of Riordan~\cite{Riordan22} and, in the triangle
case, Heckel~\cite{HeckelTriangles21}.  Kahn~\cite{Kahn22} subsequently
proved the definitive hitting-time theorem for Shamir's problem, and
Heckel, Kaufmann, M\"uller and Pasch~\cite{HKMP24} transferred it to
clique factors in the random graph process: with high probability, a
$K_r$-factor appears at the moment when the last vertex first lies in a
copy of $K_r$.  Recent work of Burghart, Heckel, Kaufmann, M\"uller and
Pasch~\cite{BHKMP24} extends the sharp-threshold theory to every strictly
$1$-balanced factor.

Another recent direction concerns the number of factors above the
threshold.  Morris and Riordan~\cite{MorrisRiordan25} study the random
hypergraph formed by the copies of $K_r$ in $G(n,p)$ and compare its
distribution with the model in which the potential copies are present
independently with probability $p^{\binom r2}$.  As an application, they
obtain a new upper bound on the number of $K_r$-factors above the
threshold.

The threshold, hitting-time and counting results above do not quantify
\emph{how unlikely} factor failure is once the density is above the
threshold.  This distinction is substantial.  A statement which holds
with high probability may discard an exceptional event of probability
$o(1)$, whereas the failure probabilities considered here can be
$\exp(-n^{\Omega(1)})$.  Our aim is to determine this much smaller
probability uniformly when a growing number of uncovered vertices is
allowed.

\subsection{The main result: rare failure above the factor threshold}

For a graph $G$ on $n$ vertices and an integer $0\le s\le n$ satisfying
$n-s\in r\mathbb Z$, let $\phi_r^s(G)$ denote the number of collections
of $(n-s)/r$ vertex-disjoint copies of $K_r$.  Thus
$\phi_r^0(G)$ is the number of $K_r$-factors.  For
$G\sim G(n,p)$, the expected number of copies of $K_r$ containing a
fixed vertex is
\[
 \binom{n-1}{r-1}p^{\binom r2}=\frac{r\mu_r}{n},
\]
where
 $$\mu_r=\mu_r(n,p):=\binom nrp^{\binom r2}.$$
Above the factor threshold this quantity is much larger than $\log n$.
Our theorem determines the exponential order of the almost-factor
failure probability throughout the sparse range, in which overlaps
between copies of $K_r$ are still negligible on the main scale.

\begin{theorem}
\label{thm:main}
Fix an integer $r\ge3$.  Let $p=p(n)$ and a non-negative integer
$s=s(n)$ satisfy
\[
 n^{-2/r}(\log n)^{1/\binom r2}\ll p\ll n^{-2/(r+1)},
 \qquad
 n-s\in r\mathbb Z,
 \qquad
 s=o(n).
\]
Then
\[
 \Pr\bigl(\phi_r^s(G(n,p))=0\bigr)
 =\exp\!\left(-\Theta_r\!\left(
 (s+1)\frac{\mu_r}n\right)\right).
\]
\end{theorem}

The theorem contains the perfect-factor case and shows that the
isolated-vertex obstruction has the correct exponential order for every
sublinear defect.  Call a vertex
$K_r$-isolated if it belongs to no copy of $K_r$, and let $Y_r$ be the
number of such vertices.  We prove, uniformly for
$1\le t=o(n)$,
\[
 \Pr(Y_r\ge t)
 =\exp\!\left(-(1+o(1))t\binom{n-1}{r-1}p^{\binom r2}\right).
\]
Thus $s+1$ $K_r$-isolated vertices give the lower bound in
Theorem~\ref{thm:main}, with the exact leading constant for this local
event.  The upper bound shows that no other obstruction changes the
exponential order.  In particular, each additional uncovered vertex
costs an exponent of order $\mu_r/n$.

The upper bound on $p$ marks a change of scale.  At
$p=\Theta( n^{-2/(r+1)})$, one has $\mu_r/n=\Theta_r(np)$.  For
$n^{-2/(r+1)}\ll p=o(1)$, forcing a fixed set of $s+1$ vertices to be
isolated has probability $\exp(-(1+o(1))(s+1)np)$ and already rules out
a matching covering $n-s$ vertices.  Since $np=o(\mu_r/n)$, this is a
cheaper obstruction than the one studied here.  Thus the natural tail
scale is no longer $(s+1)\mu_r/n$, and we do not consider this denser
range.

Taking $s=0$ gives the following consequence.

\begin{corollary}
\label{cor:factor-tail}
Fix $r\ge3$ and suppose that $r\mid n$.  If
 $n^{-2/r}(\log n)^{1/\binom r2}\ll p\ll n^{-2/(r+1)},$
then
\[
 \Pr\bigl(G(n,p)\text{ has no }K_r\text{-factor}\bigr)
 =\exp\!\left(-\Theta_r\!\left(\frac{\mu_r}n\right)\right).
\]
\end{corollary}

It is useful to compare this result with the universal theory of
expectation thresholds.  The Kahn--Kalai conjecture~\cite{KK07}, proved
by Park and Pham~\cite{PP24}, bounds the threshold of an increasing
family in terms of its expectation threshold and the logarithm of the
largest minimal member; the earlier fractional theorem of Frankston,
Kahn, Narayanan and Park~\cite{FKNP21} has many factor applications.
Bell~\cite{Bell23} obtained the optimal dependence on the target error
probability: in standard notation, if a family with minimal members of
size at most $\ell$ is not $q$-small, then density
$O(q\log(\ell/\varepsilon))$ makes the family occur with probability at
least $1-\varepsilon$.
For the increasing family $\mathcal F_r$ of edge sets containing a
$K_r$-factor, one has
$\ell(\mathcal F_r)=\binom r2n/r=\Theta_r(n)$ and
$q(\mathcal F_r)\ge c_rn^{-2/r}$.  Indeed, the copies of $K_r$ through
one fixed vertex meet every factor, and their total $q$-weight is at
most $1/2$ when $q=c_rn^{-2/r}$.  Consequently, a direct application at density $p$ gives an
error bound of the form
$\varepsilon=\ell(\mathcal F_r)\exp(-cp/q(\mathcal F_r))$.
Whenever this bound is non-trivial, its exponent is at most
$O_r(pn^{2/r})$, with a further cost of order $\log n$.
In particular, it gives no non-trivial bound in the lower part of our
range where $pn^{2/r}\ll\log n$.
By contrast, Corollary~\ref{cor:factor-tail} gives the exponent
$\mu_r/n=\Theta_r((pn^{2/r})^{\binom r2})$ throughout the stated range.
This comparison shows that rare factor failure contains more local structural
information than a general threshold theorem.

\subsection{A structural consequence: remainders of maximal clique matchings}

The dependence on the uncovered set has a structural consequence.  A
$K_k$-matching $\mathcal M$ is \emph{maximal} if no further copy of
$K_k$ can be added to it.  We write $V(\mathcal M)$ for its covered
vertex set and
$H-V(\mathcal M):=H[V(H)\setminus V(\mathcal M)]$.

\begin{theorem}
\label{thm:maximum-clique-remainder}
Fix an integer $k\ge4$, and let $q=q(n)$ satisfy
\[
 n^{-2/(k-1)}(\log n)^{1/\binom{k-1}{2}}
 \ll q\ll n^{-2/k}.
\]
If $H\sim G(n,q)$, then there is a function $\zeta=\zeta(n)\to0$ such
that, with high probability, every maximal $K_k$-matching
$\mathcal M$ in $H$ has the following property: the graph
$H-V(\mathcal M)$ contains a $K_{k-1}$-matching which leaves at most
$\zeta\sqrt{\mu_k(n,q)}$ vertices uncovered.
\end{theorem}

The conclusion holds simultaneously for every maximal matching, and in
particular for every maximum matching.  This proves the structural
conjecture proposed in our earlier paper~\cite{Yan24} throughout the
full window above.  A naive union bound over all $K_k$-matchings works
when $k\ge5$, but at $k=4$ it loses a constant factor on the natural
scale.  Maximality supplies the missing saving: the remainder is
$K_k$-free, and the probability of this additional event cancels the
entropy of choosing the matching.  This is one reason for proving the
factor tail in an iterable one-root form rather than only as an
unconditional estimate.

\subsection{An independent fluctuation theorem for clique matchings}

For a graph $H$ and an integer $k\ge2$, let $\nu_k(H)$ denote the
maximum number of vertex-disjoint copies of $K_k$ in $H$.  The following
fluctuation theorem is independent of the factor-tail argument and holds
throughout the full sparse clique-matching window.

\begin{theorem}
\label{thm:matching-clt-intro}
Fix $k\ge2$, let $q=q(n)$ satisfy
 $n^{-2/(k-1)}\ll q\ll n^{-2/k},$
and let $H\sim G(n,q)$.  Then
\[
 \frac{\nu_k(H)-\E\nu_k(H)}{\sqrt{\mu_k(n,q)}}
 \xrightarrow{\mathrm d}\mathcal N(0,1),
 \qquad
 \Var\bigl(\nu_k(H)\bigr)\sim\mu_k(n,q).
\]
\end{theorem}

Let $X_k$ be the total number of copies of $K_k$ and set
$D_k:=X_k-\nu_k$.  The conflict graph on the copies of $K_k$ identifies
$D_k$ with a minimum vertex-cover number.  Near the upper end of the
window, the mean of $D_k$ need not be smaller than
$\sqrt{\mu_k(n,q)}$, so an uncentred comparison between $X_k$ and
$\nu_k$ is insufficient.  The key point is instead that
\[
 \Var(D_k)=o(\mu_k(n,q)).
\]
We prove this by an Efron--Stein edge-influence estimate and a count of
intersecting clique pairs.  For $k\ge3$, Ruci\'nski's clique-count
central limit theorem~\cite{Ruc88} then transfers to $\nu_k$ by
Slutsky's theorem.

For $k=2$, the same collision-loss estimate is combined with the
binomial central limit theorem for the edge count.  Thus
Theorem~\ref{thm:matching-clt-intro} includes the ordinary matching
number throughout
\[
 n^{-2}\ll q\ll n^{-1}.
\]
Equivalently, writing $q=c_n/n$, this is the regime
$c_n\to0$ and $nc_n\to\infty$.  The case $k=2$ is complementary to the
recent theorem of Glasgow, Kwan, Sah and Sawhney
~\cite{GlasgowKwanSahSawhney25}.  They prove a central limit theorem for
the ordinary matching number when $q=c/n$ with fixed $c>0$, as well as
in the corresponding fixed-edge model, and the variance there has
order $n$.  At that density our estimate gives only
$\Var(D_2)=O(n)$, rather than $o(n)$.  In the fixed-edge model $X_2$ is
deterministic, so all non-trivial fluctuations of the matching number
come from $D_2$.  Their limit is governed by the Karp--Sipser process,
whereas our dilute-conflict argument transfers the Gaussian fluctuation
of the object count after showing that the \emph{centred} conflict loss
is negligible.

Very recently, Cohen Antonir, Hoshen and
Zhukovskii~\cite{CohenAntonirHoshenZhukovskii26} proved a local central
limit theorem for clique counts throughout the essentially optimal
range
\[
 n^{-2/(k-1)}\ll q\le\frac12,
\]
for every fixed $k\ge3$.  Thus the input variable $X_k$ has a local
Gaussian approximation throughout our sparse clique-matching window,
not merely a central limit theorem.  Our conclusion concerns the
optimisation variable $\nu_k$.  The estimate
$\Var(D_k)=o(\mu_k(n,q))$ is enough to transfer the ordinary central
limit theorem from $X_k$ to $\nu_k$, but it does not by itself transfer
a local limit theorem, which requires control at the lattice scale.

\subsection{The chromatic-number application}

The second application concerns the chromatic-number problem which
originally motivated the factor result.  

Understanding the distribution
of the chromatic number of a random graph is one of the central problems
in probabilistic combinatorics.  For constant edge probability its
first-order asymptotics were determined by Bollob\'as~\cite{Bol88}, while
Shamir and Spencer~\cite{SS87} initiated the study of concentration.
Much sharper concentration is known in several sparse
ranges~\cite{AK97,AN05,COPS08}; by contrast, the work of
Heckel~\cite{Hec21} and Heckel--Riordan~\cite{HR23} shows that bounded
concentration fails at constant density.  In a simplified form, the
Zigzag Conjecture of Bollob\'as, Heckel, Morris, Panagiotou, Riordan and
Smith \cite{HR23} predicts that two competing sources govern the fluctuations of
$\chi(G(n,1/2))$: the numbers of independent sets of maximum size and of
independent sets one vertex smaller.  Their relative strengths vary
with $n$, producing the conjectured zigzag in the concentration
scale.

In the very dense regime, the problem may be formulated in the sparse
complement.  If $H\sim G(n,p)$, then a proper colouring of
$G(n,1-p)=\overline H$ is exactly a partition of $V(H)$ into cliques of
$H$.  Surya and Warnke~\cite{SW24} conjectured that, in the window
\[
 n^{-2/r}(\log n)^{1/\binom r2}\ll p
 \leq (1+o(1))n^{-2/(r+1)},
\]
the natural fluctuation scale of $\chi(G(n,1-p))$ is
$\sqrt{\mu_{r+1}}$.  Their heuristic is that a near-optimal colouring
first uses as many vertex-disjoint independent sets of maximum size
$r+1$ as possible, and then covers almost all remaining vertices by
independent sets of size $r$.  Since the number of independent
$(r+1)$-sets fluctuates on the scale $\sqrt{\mu_{r+1}}$, one expects
these fluctuations to be inherited by the chromatic number.

The first non-trivial case $r=2$, corresponding to triangles in the
complement, was studied in~\cite{Yan24}.  Its main structural theorem
says that after removing a largest triangle matching, the remainder has
a near-perfect ordinary matching.  For $r\ge3$,
Theorem~\ref{thm:maximum-clique-remainder}, with $k=r+1$, gives the
corresponding higher-order structure throughout the full window.
Theorem~\ref{thm:matching-clt-intro}, again with $k=r+1$, supplies the
fluctuation input in every case.  Combining these results yields the
following Gaussian strengthening of the Surya--Warnke prediction.

\begin{theorem}
\label{thm:chromatic-clt}
Fix $r\ge2$, and let $p=p(n)$ satisfy
$
 n^{-2/r}(\log n)^{1/\binom r2}\ll p\ll n^{-2/(r+1)}.
$
Then
\[
 \frac{\chi(G(n,1-p))-\E\chi(G(n,1-p))}
      {\sqrt{\mu_{r+1}(n,p)}/r}
 \xrightarrow{\mathrm d}\mathcal N(0,1),
 \qquad
 \Var\bigl(\chi(G(n,1-p))\bigr)
 \sim\frac{\mu_{r+1}(n,p)}{r^2}.
\]
\end{theorem}

For $r=2$, this removes the artificial $n^{-7/9}$ restriction from the
triangle-window central limit theorem in~\cite{Yan24}; for $r\ge3$, it
follows from the new maximal-remainder theorem.  Thus the
Surya--Warnke prediction holds in the stronger Gaussian sense throughout
the interior of every clique window with $r\ge2$.
The theorem does not cover the boundary regime $p=\Theta(n^{-2/(r+1)})$.

The theorem identifies not only the scale and limiting law, but also the
random variable which creates the fluctuations.  
Let $X_{r+1}=X_{r+1}(G(n,p))$ denote the number of copies of $K_{r+1}$ in
$G(n,p)$, equivalently the number of independent sets of size $r+1$ in
$G(n,1-p)$.

\begin{corollary}
\label{cor:chromatic-count-equivalence}
Fix $r\ge2$, let $p=p(n)$ satisfy
$n^{-2/r}(\log n)^{1/\binom r2}\ll p\ll n^{-2/(r+1)},$ and let $\chi:=\chi(G(n,1-p))$ and $\mu_{r+1}=\mu_{r+1}(n,p)$. Then
\[
 \left\|
 (\chi-\E \chi)+\frac{X_{r+1}-\mu_{r+1}}{r}
 \right\|_2=o(\sqrt\mu_{r+1}),
\]
where $\|Z\|_2=(\E Z^2)^{1/2}$. In particular,
\[
 \operatorname{Corr}(\chi,X_{r+1})\longrightarrow-1.
\]
\end{corollary}

Corollary~\ref{cor:chromatic-count-equivalence} makes the
Surya--Warnke heuristic precise at the level of centred fluctuations.
The structural reduction first compares the chromatic number with the
maximum (clique) matching number $\nu_{r+1}$, rather than with the total count
$X_{r+1}$.  Indeed, the mean collision loss
$\E(X_{r+1}-\nu_{r+1})$ need not be negligible on the
$\sqrt{\mu_{r+1}}$ scale near the upper end of the window.  The key
phenomenon is subtler: the \emph{centred} collision loss has variance
$o(\mu_{r+1})$.  Hence all first-order randomness survives the passage
from the total number of maximum independent sets to the largest
disjoint family, and then from that family to an optimal colouring.

This gives both a parallel with and a contrast to the Zigzag picture.
In the constant $p$ setting, the maximum independent-set size grows
with $n$, and the counts at two adjacent sizes are conjectured to be
competing fluctuation sources.  In a fixed very-dense clique window, by
contrast, the maximum size is the constant $r+1$: larger independent
sets are negligible, while the $r$-set packing of the remainder has
$o(\sqrt{\mu_{r+1}})$ fluctuations.  A single source therefore
remains, and it is asymptotically Gaussian and perfectly negatively
correlated with the chromatic number.

\subsection*{Organisation of the paper}

Section~\ref{sec:overview} outlines the proof, and
Section~\ref{sec:local-obstruction} gives the lower bound.
Sections~\ref{sec:one-root}--\ref{sec:fixed-rank} prove the one-root
theorem: factor failure leads to a stopping pattern, then to a rooted
fingerprint, and finally to a conditional rare event.
Section~\ref{sec:factor-tail-proof} iterates this theorem and proves
the maximal-remainder result.  Sections~\ref{sec:clique-matching-clt}
and~\ref{sec:chromatic} treat the matching and chromatic fluctuations,
and Section~\ref{sec:remark} gives further questions.  The appendices
contain the averaged coefficient estimate and the derivative calculations
used in the random-graph bounds.

\section{Proof overview}
\label{sec:overview}

The lower bound comes from $s+1$ vertices which lie in no copy of
$K_r$.  Harris's and Janson's inequalities give the probability of
this event, and the assumption $\mu_r/n\gg\log n$ absorbs the cost of
choosing the vertices.  We turn to the upper bound, where the main
point is to obtain the same exponential cost for each uncovered vertex.

\subsection{The one-root reduction}

The proof follows the factor-deletion method of Johansson, Kahn and
Vu, but needs estimates that remain valid after later vertex deletions.
We therefore assign one set of independent birth times to the edges
of $K_n$ and use its restriction on every induced graph.  Write
$\mathscr F(x)$ for all information outside the star of $x$.
Theorem~\ref{prop:one-root-witness} shows that, for every admissible
$W$ of order $n-o(n)$, factor failure on a common decreasing event
$\cU$ is covered by a global error $\cM_W$ and polynomially many
local events $\mathcal E_{W,\alpha}$.  The global error has probability
$\exp(-\Omega(\mu_r))$, while each local event has a root
$x=x_{W,\alpha}$ and satisfies
\[
 \Pr(\mathcal E_{W,\alpha}\mid\mathscr F(x))
 \le \exp(-c\mu_r/n)
 \qquad\text{almost surely}.
\]
We call such an event a single-root witness.  Its conditional bound
is the part of the argument that will allow iteration.

The one-root reduction has three steps.

\emph{From factor failure to a stopping pattern.}
Read the random edge ordering on $W$ backwards, and let $F_j$ be the
number of factors at rank $j$.  One deletion destroys the proportion
$\xi_j=1-F_{j-1}/F_j$ of the current factors.  We follow the exact
logarithmic loss $-\log(1-\xi_j)$ while the factor count stays above a
barrier $\cA$, the rooted clique counts satisfy $\cR$, and the factor
weights satisfy a balance condition $\cB$.
An averaged exponential second-moment bound for the edge loads gives
a martingale error of probability $\exp(-\Omega(\mu_r))$.  Outside
this error, factor failure forces $\cR$ or $\cB$ to fail while
$\cA$ still holds.  The median condition $\cC$ then separates the two
possible forms of balance failure.

\emph{From a stopping pattern to a rooted fingerprint.}
If $\cC$ holds but $\cB$ fails, a large factor weight spreads over
a dense family of $r$-sets.  Taking successive vertex links turns the
resulting clique deficit into a lower-tail fingerprint: one root
completes too few members of an outside family of $(\ell-1)$-cliques.
A lower failure of $\cR$ gives the same type of fingerprint.

If $\cC$ fails, an $(r-1)$-set has one heavy extension and many light
extensions.  An entropy argument retains a linear number of possible
light roots.  We truncate the factor weights at a fixed scale and
record the rank outside the two roots.  The heavy and light extensions
then give two values of the same polynomial, with coefficients in
$[0,1]$, that differ by a fixed factor.  This is the two-star fingerprint.

\emph{From a fingerprint to a single-root witness.}
After the outside birth times and the root degree are fixed, the root
neighbourhood is a uniform set.  Janson's inequality handles the
lower-tail fingerprint.  For the two-star fingerprint, the averaged
coefficient bound lets us choose a regular light root among the
linearly many candidates.  We colour its coordinates and expose one
colour class at a time.  The polynomial becomes a martingale with
linear increments, whose exponential moments are controlled by the
same averaged bound.  Comparing the two conditional means then forces
an upper deviation at the light root or a lower deviation at the heavy
root.  Both have the required conditional probability.

The common event $\cU$ supplies the edge-load, rooted-degree and
overlap estimates used in these steps.  Its probability need only be
bounded below by a positive constant, since Harris's inequality removes
it at the end.  The use of average exponential moments, rather than
only the largest coefficient, is what permits the full range
$p\ll n^{-2/(r+1)}$.

\subsection{Iteration and applications}

If no $K_r$-matching covers $n-s$ vertices, then deleting any suitable
set of at most $s$ vertices still leaves a graph with no factor.
We apply the one-root theorem $s+1$ times, removing each selected root
before continuing.  Every later event is measurable outside each
earlier root star, so the conditional bounds multiply and give
$\exp(-\Omega((s+1)\mu_r/n))$ for a fixed chain.
There are only $\exp(O((s+1)\log n))$ labelled chains, and this cost
is absorbed by $\mu_r/n\gg\log n$.  The global errors are handled
separately by a union bound over the possible vertex sets.

The same iteration can retain an event supported on the final
remainder.  Requiring this remainder to be $K_{r+1}$-free gives the
maximal-matching theorem; at the triangle-factor window, this extra
probability is needed to cancel the cost of choosing the matching.
Independently, an Efron--Stein estimate shows that the centred loss
between the clique count and the maximum clique-matching number has
negligible variance.  Together with the remainder theorem, this
identifies the first-order chromatic fluctuation and proves its
Gaussian limit.

\subsection{Comparison with the Johansson--Kahn--Vu argument}

The proof refines the Johansson--Kahn--Vu deletion argument.  We keep
its main structure: follow the factor count under edge deletion, use
factor weights and rooted clique counts, and separate balance failure
by a median condition.  The difference is in the probability estimate
and the form in which it is obtained.  A threshold proof may discard
an event of probability $o(1)$, whereas our error must be exponentially
small and remain controlled after further vertex deletions.

First, we keep all induced graphs in one birth-time space.  Every
deletion chain is a restriction of the same edge ordering, and
$G_p[W]$ is read at its exact random rank.  This avoids a transfer
between fixed-edge and binomial models.  It also makes every later
process on a set avoiding $x$ measurable outside the $x$-star, which
is needed for conditioning and iteration.

Second, we cover the entire failure event on $\cU$ by a global error
and polynomially many single-root witnesses.  This is the main
additional step in the one-root argument.  Each possible failure of
balance or rooted regularity is first reduced to a fingerprint, then
to an event $\mathcal E$ with a root $x$ and the conditional bound in
Theorem~\ref{prop:one-root-witness}.  The two-star fingerprint must
therefore be split into a deviation at one of its two roots; the
vertices used in the link recursion remain only as label data.
The averaged exponential second-moment estimates give the required
precision both for the deletion martingale and for the rooted
exposure.  Thus the witness records not only the probability of
failure, but also the single star on which its remaining randomness
lies.

Finally, the one-root bound is designed to be iterated.  Heredity lets
us remove one root at a time, while measurability in the common
birth-time space allows conditional expectation to multiply the
bounds.  This gives the factor $s+1$ in the exponent of
Theorem~\ref{thm:main}; the polynomial number of labels has a smaller
cost.  The same argument can keep an event on the final remainder,
such as being $K_{r+1}$-free.  This extra step connects the factor tail
to maximal matchings and the chromatic number.

\section{The local obstruction}
\label{sec:local-obstruction}

The lower bound follows by requiring $s+1$ vertices to lie in no copy
of $K_r$.  We prove the leading constant for this event as well.

\begin{proposition}
\label{prop:isolated-tail}
Fix $r\ge3$ and suppose
$n^{-2/r}(\log n)^{1/\binom r2}\ll p\ll n^{-2/(r+1)}$.
If $Y$ is the number of vertices of $G(n,p)$ lying in no copy of
$K_r$, then, uniformly for integers $1\le t=o(n)$,
\[
 \Pr(Y\ge t)=
 \exp\!\left(-(1+o(1))t\binom{n-1}{r-1}p^{\binom r2}\right).
\]
Taking $t=s+1$ gives the lower bound in Theorem~\ref{thm:main}.
\end{proposition}

We use the following forms of Harris's and Janson's inequalities.
The first also applies to the independent birth times used later.

\begin{lemma}[Harris]
\label{lem:harris}
Let $A$ and $B$ be events in a product probability space with the
coordinatewise order.  If both are increasing or both are decreasing,
then
\[
 \Pr(A\cap B)\ge\Pr(A)\Pr(B).
\]
If one is increasing and the other is decreasing, the inequality is
reversed.
\end{lemma}

\begin{lemma}[Janson]
\label{lem:janson}
Let $(\xi_i)_{i\in I}$ be independent Bernoulli variables and let
$\mathcal H$ be a family of nonempty subsets of $I$.  Set
$I_A=\prod_{i\in A}\xi_i$, $X=\sum_{A\in\mathcal H}I_A$,
$\mu=\E X$, and
\[
 \Delta=\sum_{\substack{A,B\in\mathcal H\\
                        A\ne B,\ A\cap B\ne\varnothing}}
          \E[I_AI_B],
\]
where the pairs are ordered.  For $0<a<1$,
\[
 \Pr(X=0)\le e^{-\mu+\Delta/2},
 \qquad
 \Pr(X\le(1-a)\mu)\le
 \exp\!\left(-\frac{a^2\mu^2}{2(\mu+\Delta)}\right).
\]
\end{lemma}

We apply these bounds to the copies meeting a fixed set of vertices.

\begin{proof}[Proof of Proposition~\ref{prop:isolated-tail}]
Fix a $t$-set $T$ and let $\cI_T$ be the event that no copy of $K_r$
meets $T$.  There are
$N_T=\binom nr-\binom{n-t}r=(1+o(1))t\binom{n-1}{r-1}$ candidate
sets, uniformly for $t=o(n)$.  Put $\lambda_T=N_Tp^{\binom r2}$.
By Harris's inequality,
$\Pr(\cI_T)\ge(1-p^{\binom r2})^{N_T}
=\exp(-(1+o(1))\lambda_T)$.

For the reverse bound, fix one candidate.  There are at most
$C_rn^{r-h}$ other $r$-sets meeting it in exactly $h$ vertices, so
\[
 \frac{\Delta_T}{\lambda_T}
 \le C_r\sum_{h=2}^{r-1}n^{r-h}p^{\binom r2-\binom h2}=o(1).
\]
Indeed, the term with index $h$ equals
$n^{(r-h)(2-h)/(r+1)}
(pn^{2/(r+1)})^{\binom r2-\binom h2}$, which tends to zero for every
$2\le h<r$.  Janson's inequality therefore gives
$\Pr(\cI_T)=\exp(-(1+o(1))\lambda_T)$.

One fixed set $T$ gives the lower bound for $\Pr(Y\ge t)$.
For the upper bound, sum over the $\binom nt$ choices.
Since $\lambda_T\gg t\log n$ and
$\log\binom nt\le t\log(en)$, this does not change the leading
exponent.  Finally, $s+1$ such vertices prevent any matching from
covering $n-s$ vertices.
\end{proof}
\section{The one-root reduction: statement and setup}
\label{sec:one-root}

We now state the conditional estimate that will be iterated in
Section~\ref{sec:factor-tail-proof}.  Assign independent birth times
$\tau_f\sim\operatorname{Unif}[0,1]$ to all $f\in\binom{[n]}2$,
and let $G_q$ contain the edges with $\tau_f\le q$.
Thus $G_q\sim G(n,q)$.  On each vertex set $W$, ordering the edges by
these same birth times gives the exact-rank graphs $H_j(W)$, and
$G_p[W]=H_{J_p}(W)$ for $J_p=|E(G_p[W])|$.

For $x\in W$, let $\mathscr F(x)$ be the sigma-field generated by all
birth times outside the star of $x$, and set
\[
 \mathscr F_W(x):=\sigma(\tau_f:f\in\tbinom W2,\ x\notin f),
 \qquad
 \mathscr G_W:=\sigma(\tau_f:f\in\tbinom W2).
\]
For every $\mathscr G_W$-measurable event $\mathcal E$, independence
of the coordinates outside $W$ gives
$\Pr(\mathcal E\mid\mathscr F(x))
=\Pr(\mathcal E\mid\mathscr F_W(x))$ almost surely.
Also, every event determined by an induced graph whose vertex set
omits $x$ is $\mathscr F(x)$-measurable.

\begin{definition}[Single-root witness]
\label{def:single-root-witness}
A pair $(\mathcal E,x)$ is a \emph{single-root witness with parameter
$c>0$, local to $W$} if $x\in W$, $\mathcal E\in\mathscr G_W$, and
\[
 \Pr(\mathcal E\mid\mathscr F(x))\le \exp(-c\mu_r/n)
 \quad\text{almost surely}.
\]
\end{definition}

The estimate is pointwise in the outside configuration.  A witness
may include conditions on this configuration: where they fail the
event is empty, and where they hold we estimate the remaining
root-star event.  Conditions involving the root are kept inside the
event being estimated, rather than added to the conditioning.

\begin{theorem}[One-root witness]
\label{prop:one-root-witness}
Fix $r\ge3$, let $b=b(n)\ge0$ satisfy $b=o(n)$, and suppose
$n^{-2/r}(\log n)^{1/\binom r2}\ll p\ll n^{-2/(r+1)}$.
There are a common decreasing event $\cU$ with
$\Pr(\cU)\ge c_r>0$ and constants $c,C>0$, depending only on $r$,
with the following property.  For every $W\subseteq[n]$ with
$|[n]\setminus W|\le b$ and $r\mid|W|$, there is an event $\cM_W$
of probability at most $\exp(-c|W|\mu_r/n)$ such that
\begin{equation}
 \{G_p[W]\text{ has no }K_r\text{-factor}\}\cap\cU
 \subseteq \cM_W\cup
 \bigcup_{\alpha\in\mathfrak C_W}\mathcal E_{W,\alpha},
 \label{eq:witness-cover}
\end{equation}
where $|\mathfrak C_W|\le n^C$ and every
$(\mathcal E_{W,\alpha},x_{W,\alpha})$ is a single-root witness
with parameter $c$, local to $W$.
\end{theorem}

Sections~\ref{sec:deletion-process} and~\ref{sec:fingerprints} reduce
factor failure to a stopping pattern and then to a rooted fingerprint.
Section~\ref{sec:upper-regularity} supplies the common random-graph
bounds, and Section~\ref{sec:fixed-rank} estimates the fingerprints
and completes the theorem.  Besides $\cU$, we use an exact-rank
comparison event $\cO_W$ for each fixed $W$.  Its failure probability
is $\exp(-\Omega(p|W|^2))$ and is included in $\cM_W$.
The distinction is useful: $\cU$ is retained throughout the argument,
whereas all events $\cM_W$ are treated by a union bound.

\section{From factor failure to a stopping pattern}
\label{sec:deletion-process}

Fix $b=o(n)$ from Theorem~\ref{prop:one-root-witness} and
$W\subseteq[n]$ with $|[n]\setminus W|\le b$ and $r\mid |W|$.
Write $m=|W|$ and $N=\binom m2$.  Order the edges of $K_W$ as
$e_1,\ldots,e_N$ by increasing birth time, and set
$H_j=H_j(W)=\{e_1,\ldots,e_j\}$ and $q_j=j/N$.
We read this chain backwards, deleting a uniform present edge at each
step, until $G_p[W]=H_{J_p}$, where $J_p=|E(G_p[W])|$.

The reduction below says that factor failure forces a local condition
to fail while the factor count is still large.  We define its four
conditions immediately after the statement.

\begin{proposition}
\label{prop:stopping-reduction}
There is an event $\cM_W$ with
$\Pr(\cM_W)\le\exp(-c|W|\mu_r(n,p)/n)$ such that, on $\cU$, if
$G_p[W]$ has no $K_r$-factor, then either $\cM_W$ occurs or, at some
rank $j$ before the $p$-endpoint, $\cA(H_j,q_j)$ holds and one of the
following disjoint patterns occurs:
\[
 \overline{\cR(H_j,q_j)},\qquad
 \cR(H_j,q_j)\cap\overline{\cB(H_j)}\cap\overline{\cC(H_j)},\qquad
 \cR(H_j,q_j)\cap\overline{\cB(H_j)}\cap\cC(H_j).
\]
\end{proposition}

\subsection{The control conditions}

For a graph $H$, let $K_\ell(H)$ be its family of $\ell$-cliques and
let $d_{K_\ell}(v;H-R)$ count those containing $v$ in $H-R$.
Write $\Phi(H)$ for the number of $K_r$-factors and
$w_H(S)=\Phi(H-S)$ for $S\in\binom Wr$.  A factor together with one
of its blocks can be counted in two ways, giving
\[
 \sum_{K\in K_r(H)}w_H(K)=\frac mr\Phi(H).
\]
The proportion of factors lost when an edge $f$ is deleted is thus
$\sum_{K\ni f}w_H(K)/\Phi(H)$, where the sum is over cliques
containing $f$.  We control this quantity by keeping both the number
of cliques and their weights regular.

Unless stated otherwise, constants in the upper-bound proof depend
only on $r$ and are chosen compatibly throughout.  We take
$c_{\mathrm R}$ sufficiently small and $C_{\mathrm R},C_B,C_A$
sufficiently large as required below, and choose $\nu$ after the
constants in the fingerprint estimates have been fixed.
Constants in auxiliary estimates may also depend on the fixed
parameters in their statements; these are functions of $r$ in our applications.
At density label $q$, the \emph{factor-count barrier} is
\[
 \cA(H,q):\qquad
 \Phi(H)\ge e^{-C_Am}\left(\frac{\mu_r(n,q)}n\right)^{m/r}.
\]
The \emph{rooted-regularity} condition is
\[
 \cR(H,q):\qquad
 c_{\mathrm R}\frac{\mu_r(n,q)}n
 \le d_{K_r}(v;H-R)\le C_{\mathrm R}\frac{\mu_r(n,q)}n
\]
for every $v\in W$ and $R\subseteq W\setminus\{v\}$ with
$|R|\le3r$.  The bounded deleted set is needed for the link arguments
in Section~\ref{sec:fingerprints}.
The \emph{factor-weight balance} condition is
\[
 \cB(H):\qquad
 \max_{K\in K_r(H)}w_H(K)
 \le\frac{C_B}{|K_r(H)|}\sum_{K\in K_r(H)}w_H(K).
\]
Finally, for $Y\in\binom W{r-1}$ put
$\psi_H(Y)=\max_{x\in W\setminus Y}w_H(Y\cup\{x\})$.
The \emph{median condition} is
\[
 \cC(H):\qquad
 \psi_H(Y)\le\max\{n^{-2(r-1)}\Phi(H),
          2\med_{x\in W\setminus Y}w_H(Y\cup\{x\})\}
\]
for every $Y$.  Only $\cR$ and $\cB$ control the deletion process;
$\cC$ will divide balance failure into the two cases treated next.

For later estimates set $e=\binom r2$, $\eta_r(q)=n^{r-2}q^{e-1}$,
$T(q)=1+\eta_r(q)$ and $A(q)=\max\{1,(q/p)^{1/2}\}$.
Here $\eta_r(q)$ is the expected edge--clique load, up to a constant.
We use
\begin{equation}
 \mu_r(n,q)/n=\Theta_r(nq\eta_r(q))
 \label{eq:mu-eta-relation}
\end{equation}
and, uniformly for $q\ge p/4$,
\begin{equation}
 (p/q)^eT(q)\le C_r,
 \qquad \frac{\mu_r(n,q)}{nA(q)T(q)}
             \ge c_r\frac{\mu_r(n,p)}n.
 \label{eq:scale-comparisons}
\end{equation}
Indeed, if $q=ap$, the first expression is
$a^{-e}+\eta_r(p)a^{-1}$; for $a\ge1$, the ratio in the second is
$a^{e-1/2}/(1+\eta_r(p)a^{e-1})$ up to a constant.
Both claims follow from $\eta_r(p)=o(1)$; bounded $a\ge1/4$ causes
no change.  We also have $\mu_r(n,p)/n\gg\log n$.

\subsection{The logarithmic-loss martingale}

Proposition~\ref{prop:uniform-input}, proved in the next sections,
supplies the maximum and averaged edge-load bounds on $\cU\cap\cO_W$.
We first show that these bounds prevent the factor count from crossing
its barrier while $\cR$ and $\cB$ hold.

Let
\[
 \mathcal E_W=\{pN/2\le J_p\le2pN\},
 \qquad \Pr(\overline{\mathcal E_W})\le e^{-cpm^2},
\]
where the estimate is Chernoff's inequality.  Set
$j_0=\lceil pN/3\rceil$, $j_\star=\max\{J_p,j_0\}$ and
$\gamma_j=me/(rj)$.
Reveal $J_p$ first and use the filtration
$\mathcal F_t=\sigma(J_p,e_N,\ldots,e_{N-t+1})$.
The rank $J_p$ is independent of the edge ordering, so the next edge
is still uniform in the current graph.
Run the backward chain down to $j_\star$, stopping before the first
deletion whose source fails $\cA$, $\cR$, $\cB$ or either edge-load
bound in Proposition~\ref{prop:uniform-input}(ii).
This is a measurable stopping rule for $(\mathcal F_t)$.

At an included deletion put
$\xi_j=1-\Phi(H_{j-1})/\Phi(H_j)$, $L_j=-\log(1-\xi_j)$,
$\ell_j=\EE[L_j\mid\mathcal F_{N-j}]$ and $Z_j=L_j-\ell_j$.
The source has a positive factor count, and the estimate below gives
$\xi_j=o(1)$.  Put all four variables equal to zero after stopping.
Then $M_t=\sum_{j=N-t+1}^NZ_j$ is a martingale.  Define
\[
 \cD_W=\{\max_{0\le t\le N-j_0}M_t>m\},
 \qquad \cM_W=\cD_W\cup\overline{\mathcal E_W}\cup\cO_W^c.
\]

\begin{lemma}
\label{lem:deletion-martingale}
We have
\[
 \Pr(\cM_W)\le\exp(-c|W|\mu_r(n,p)/n).
\]
Outside $\cM_W$, suppose that whenever $\cA$ holds, $\cR$, $\cB$
and the two edge-load bounds hold at the source graph.  Then
$\Phi(G_p[W])\ge e^{-C_0|W|}(\mu_r(n,p)/n)^{|W|/r}>0$.
\end{lemma}

We begin with the one-step estimate.  For an edge $f$ in the current
graph, write $D_f=d_{K_r}(f;H)$ and
$\xi_f=1-\Phi(H-f)/\Phi(H)$.

\begin{lemma}
\label{lem:one-step}
Suppose $H=H_j(W)$ has a positive factor count, $q=j/N\ge p/3$, and
$\cR(H,q)$, $\cB(H)$ and both edge-load bounds hold.
Write $\theta_q=\mu_r(n,q)/n$ and let $f$ be uniform in $E(H)$.
Then $\EE_f\xi_f=\gamma_j$, $0\le\xi_f\le C_rD_f/\theta_q$ and
$\max_f\xi_f=o(1)$.
For $L_f=-\log(1-\xi_f)$ and $\lambda=c_r\mu_r(n,p)/n$, with
$c_r$ sufficiently small,
\[
 \EE_fL_f\le\gamma_j+2\EE_f\xi_f^2,
 \qquad
 \EE_f[L_f^2e^{\lambda L_f}]
       \le\frac{C_rA(q)\eta_r(q)T(q)}{\theta_q^2}.
\]
\end{lemma}

\begin{proof}
Every factor has $me/r$ edges, giving $\EE_f\xi_f=me/(rj)$.
Balance and the weight-sum identity give
$w_H(K)/\Phi(H)\le C_rm/|K_r(H)|=O_r(1/\theta_q)$.
Summing over the $D_f$ cliques containing $f$ proves the bound on
$\xi_f$.  The maximum-load estimate then gives
\[
 \max_f\xi_f\le C_r\frac{\eta_r(q)+\log n}{\theta_q}
       =O_r\left(\frac1{nq}+\frac{\log n}{\mu_r(n,p)/n}\right)=o(1).
\]
Thus $L_f\le\xi_f+2\xi_f^2$ and $L_f\le2\xi_f$, proving the mean
bound.  The scale comparison gives
$\lambda L_f\le C_rc_r(p/q)^eD_f\le c'_rD_f/T(q)$ after decreasing
$c_r$.  The averaged load bound now yields
\[
 \EE_f[L_f^2e^{\lambda L_f}]
 \le\frac{C_r}{\theta_q^2|E(H)|}
       \sum_{f\in E(H)}D_f^2e^{c'_rD_f/T(q)}
 \le\frac{C_rA(q)\eta_r(q)T(q)}{\theta_q^2},
\]
since $|E(H)|=\Theta(n^2q)$.
\end{proof}

\begin{proof}[Proof of Lemma~\ref{lem:deletion-martingale}]
Set $\theta_p=\mu_r(n,p)/n$ and $\lambda=c_r\theta_p$ as above.
First sum the one-step bounds over the density levels.
In a dyadic block $ap\le q_j<2ap$, $a=1,2,4,\ldots$, there are
$O(m^2ap)$ ranks, and its contribution is at most
\[
 \sum_{j:\,ap\le q_j<2ap}
 \frac{A(q_j)\eta_r(q_j)T(q_j)}{(\mu_r(n,q_j)/n)^2}
 \le\frac{C_rm}{\theta_p}
       \bigl(a^{-(e-1/2)}+\eta_r(p)a^{-1/2}\bigr).
\]
Here $\eta_r(q_j)=\Theta_r(a^{e-1}\eta_r(p))$,
$\mu_r(n,q_j)/n=\Theta_r(a^e\theta_p)$ and
$\theta_p=\Theta_r(np\eta_r(p))$.
Both dyadic sums converge, and the ranks with $p/3\le q_j<p$
contribute the same order as the first block.  Hence, along the
stopped process,
\[
 \sum_{j\ge j_0}\EE[L_j^2e^{\lambda L_j}\mid\mathcal F_{N-j}]
       =O_r(m/\theta_p),
 \qquad
 \sum_{j\ge j_0}\EE[\xi_j^2\mid\mathcal F_{N-j}]=o(m).
\]
The second claim follows from $\xi_j\le L_j$ and $\theta_p\to\infty$.

For any non-negative $X$ and sigma-field $\mathcal F$, the inequality
$e^u\le1+u+u^2e^u/2$ gives
\[
 \EE[e^{\lambda(X-\EE[X\mid\mathcal F])}\mid\mathcal F]
 \le\exp\bigl(\lambda^2\EE[X^2e^{\lambda X}\mid\mathcal F]/2\bigr).
\]
Applying this to $L_j$ shows that
\[
 \exp\left(\lambda M_t-
 \frac{\lambda^2}{2}\sum_{j=N-t+1}^N
       \EE[L_j^2e^{\lambda L_j}\mid\mathcal F_{N-j}]\right)
\]
is a non-negative supermartingale starting at one.
Its maximal inequality and the preceding sum bound give
$\Pr(\cD_W)\le\exp(-\lambda m+C_r\lambda^2m/\theta_p)
\le e^{-c_rm\theta_p}$ for small enough $c_r$.
Both $\Pr(\overline{\mathcal E_W})$ and $\Pr(\cO_W^c)$ are at most
$e^{-cpm^2}$.  Since $pm^2/(m\theta_p)=\Theta_r(1/\eta_r(p))\to\infty$,
these errors give the claimed bound on $\cM_W$.

It remains to show that $\cA$ cannot fail first.
Outside $\cM_W$, take any rank $h$ reached by the stopped chain,
including its terminal graph.  The mean estimate and the squared-loss
sum imply
$\sum_{j=h+1}^NL_j\le\sum_{j=h+1}^N\gamma_j+O_r(m)$.
Since each included deletion has $\xi_j=o(1)$, its factor count stays
positive, and the product of the surviving proportions gives
\[
 \log\Phi(H_h)
 \ge\log\Phi(K_W)-\sum_{j=h+1}^N\gamma_j-O_r(m)
 =\frac mr\log\left(\frac{\mu_r(n,q_h)}n\right)-O_r(m).
\]
For the equality, use Stirling's formula for
$\Phi(K_W)=m!/((r!)^{m/r}(m/r)!)$, the harmonic sum
$\sum_{j>h}\gamma_j=(me/r)\log(1/q_h)+O_r(m)$, and $m=n-o(n)$.
Choose $C_A$ larger than the implicit constant.  The complete graph
satisfies $\cA$, and the display excludes its first failure: the
preceding deletion would still be included and would give the required
barrier at that rank.

Under the hypotheses of the lemma, none of the other stopping
conditions can fail while $\cA$ holds.  Thus the chain reaches
$j_\star=J_p$.  Since $q_{J_p}=\Theta(p)$ on $\mathcal E_W$, the
same bound gives
$\Phi(G_p[W])\ge e^{-C_0m}(\mu_r(n,p)/n)^{m/r}>0$.
\end{proof}

\begin{proof}[Proof of Proposition~\ref{prop:stopping-reduction}]
Work on $\cU\setminus\cM_W$ and suppose that no listed pattern occurs
while $\cA$ holds.  Then $\cR$ and $\cB$ both hold there, regardless
of $\cC$.  Also $\cO_W$ holds, so
Proposition~\ref{prop:uniform-input} supplies both edge-load bounds.
Lemma~\ref{lem:deletion-martingale} gives a factor in $G_p[W]$,
proving the contrapositive.
\end{proof}
\section{From stopping patterns to rooted fingerprints}
\label{sec:fingerprints}

We now turn each stopping pattern into a local configuration, which we
call a \emph{fingerprint}.  The argument is deterministic; the
conditional probability bounds are proved in
Section~\ref{sec:fixed-rank}.
Fix $H=H_j(W)$ and $q=q_j(W)$, and assume the edge and rooted
clique-degree upper bounds in Proposition~\ref{prop:uniform-input}.

A \emph{lower-tail fingerprint} consists of a root $x$, an integer
$2\le\ell\le r$, a restricted-rank outside graph, and an
outside-determined family $\cH\subseteq K_{\ell-1}(H-x)$ such that
\[
 |\cH|\ge c_0n^{\ell-1}q^{\binom{\ell-1}{2}},
\]
but the number $X_x$ of its members completed to a $K_\ell$ through
$x$ satisfies
\[
 X_x<c_1\mu_\ell(n,q)/n.
\]
Thus the outside graph offers many extensions, but the root completes
far too few of them.

A \emph{two-star fingerprint} consists of distinct roots $x,y$ and a
homogeneous polynomial $Z(t)=\sum_{A\in\cH}c_A\prod_{u\in A}t_u$,
where $\cH\subseteq K_{r-1}(H-\{x,y\})$, $0\le c_A\le1$, and the coefficients are determined outside both
roots, such that
\[
 Z(N_H(y))\ge a_0\mu_r(n,q)/n,
 \qquad Z(N_H(y))>(1+\rho_0)Z(N_H(x)).
\]

Finally, for $W'\subseteq W$ obtained by deleting $O_r(1)$ vertices,
$z\in W'$ and a rank $h$ in the induced ordering on $K_{W'}$, put
$q_h=h/\binom{|W'|}{2}$ and define the \emph{sample anomaly}
\[
 \mathsf{Anom}(W',z,h):=
 \bigl\{|d_{H_h(W')}(z)-q_h(|W'|-1)|>\nu q_h|W'|\bigr\},
\]
where $\nu$ is the constant in Lemma~\ref{lem:ranked-star}.

\begin{proposition}
\label{prop:structural-fingerprints}
Suppose that $\cA(H_j,q)$ holds and one of the stopping patterns in
Proposition~\ref{prop:stopping-reduction} occurs.  After branching over
at most $n^{C_r}$ choices of vertices, restricted ranks, dyadic scales
and root degrees, one obtains either a sample anomaly, a lower-tail
fingerprint, or a heavy root $x$ and at least $c_{\rm L}|W|$ possible
light roots $y$, each giving a two-star fingerprint.
\end{proposition}

When the median condition holds, repeated vertex replacement spreads a
large factor weight over a dense family of $r$-sets.  Balance failure
forces this family to contain too few cliques, and vertex links locate
the deficit at one root.  When the median condition fails, entropy
retains many light roots and truncation gives the common coefficients
of a two-star fingerprint.

\subsection{Balance failure and the lower-tail fingerprint}

A family $\cZ(H)\subseteq\binom W\ell$ is \emph{self-determined} if
membership of each $S$ is determined by $H-S$.  In particular, a
threshold condition on $w_H(S)=\Phi(H-S)$ has this property.
At an exact rank, we record the restricted rank before fixing such
outside data.  Any part of the rank identity involving the eventual
root remains in its event, as in Lemma~\ref{lem:record-rank-data}.

\begin{lemma}
\label{lem:median-replacement}
If $\cR(H,q)$ and $\cC(H)$ hold but $\cB(H)$ fails, there is a dyadic
value $J$ and a self-determined family $\cZ_J\subseteq\binom Wr$ with
$|\cZ_J|\ge\gamma |W|^r$ and
$|\cZ_J\cap K_r(H)|<c|W|^rq^{\binom r2}$, where $\gamma,c>0$ depend on $r$.
\end{lemma}

\begin{proof}
Set $m=|W|$ and $M=\max_{K\in K_r(H)}w_H(K)$.
The weight-sum identity and rooted regularity give
$M\ge cn\Phi(H)/\mu_r(n,q)$.  Since $\mu_r(n,q)/n\le n^{r-1}$,
every constant multiple of $M$ exceeds the exceptional threshold
$n^{-2(r-1)}\Phi(H)$ in the median condition, for large $n$.

Start with an ordered clique of weight $M$ and replace its $r$
coordinates in order.  If the current set has weight at least
$2^{-(i-1)}M$, the median condition gives at least $(m-r+1)/2$
choices for the next coordinate with new weight at least $2^{-i}M$.
The coordinates remain distinct.  Each final $r$-set has at most $r!$
representations, so at least $\gamma m^r$ sets have weight at least
$\zeta M$, for constants $\gamma,\zeta>0$ depending only on $r$.

Choose a dyadic $J$ with $J\le M<2J$ and let
$\cZ_J=\{S:w_H(S)>\zeta J/2\}$.  This family is self-determined and
has the required size.  If it contained at least
$cm^rq^{\binom r2}$ cliques, each would have weight at least
$\zeta M/4$.  As $|K_r(H)|=\Theta_r(m^rq^{\binom r2})$, their total
weight would imply $\cB(H)$ for sufficiently large $C_B$.
Thus the clique count is below the claimed bound.  There are only
$O(n\log n)$ possible values of $J$, since factor counts lie between
$1$ and $n!$.
\end{proof}

It remains to locate the clique deficit at one root.  The following
link argument also handles a lower failure of rooted regularity.

\begin{lemma}
\label{lem:recursive-fingerprint}
Let $1\le\ell\le r$, let $\gamma>0$, and let
$\cZ\subseteq\binom W\ell$ be self-determined.  If
\[
 |\cZ|\ge\gamma |W|^\ell,
 \qquad
 |\cZ\cap K_\ell(H)|<c_\gamma |W|^\ell q^{\binom\ell2},
\]
then at most $\ell-1$ successive vertex links give a sample anomaly or
a lower-tail fingerprint of uniformity between $2$ and $\ell$.
The choices of vertices and restricted ranks number $n^{O_r(1)}$.
\end{lemma}

\begin{proof}
We induct on $\ell$.  The assumptions are impossible for $\ell=1$,
since every singleton is a clique.  For $\ell\ge2$, write
$\cZ_x=\{A:A\cup\{x\}\in\cZ\}$.
The identity $\sum_x|\cZ_x|=\ell|\cZ|$ gives a set $X_0$ of
$\Omega_\gamma(n)$ vertices whose links have size at least
$c'_\gamma n^{\ell-1}$.

For each $x\in X_0$, consider the branch with $h_x=j-d_H(x)$, so that
$H-x=H_{h_x}(W\setminus\{x\})$.  If the degree of $x$ is abnormal,
we have a sample anomaly.  Otherwise the restricted density
$q_x=h_x/\binom{|W|-1}{2}$ satisfies
\[
 q_x-q=\frac{2(q(|W|-1)-d_H(x))}{(|W|-1)(|W|-2)}
          =O(\nu q/n).
\]
For fixed $h_x$, the restricted graph uses only non-$x$ birth times.
If the eventual root is $x$, keep the identity $h_x=j-d_H(x)$ in its
event.  For any other eventual root $z$, first record its total degree
at each intermediate rank.  Lemma~\ref{lem:record-rank-data} then makes
the graph away from $z$ an outside function; only the recorded state
of $zx$ remains in the degree of $x$.  Keep all rank and degree
identities, and these edge-state conditions, in the root event.

Let $\cH_x=\cZ_x\cap K_{\ell-1}(H-x)$.
If some $x\in X_0$ has
$|\cH_x|<c''_\gamma n^{\ell-1}q^{\binom{\ell-1}{2}}$, apply
induction to $\cZ_x$ on the branch $h_x$.
Indeed, membership of $A$ is determined by
$H-(A\cup\{x\})=(H-x)-A$, and the restricted density differs from
$q$ by a factor $1+O(\nu/n)$.

Otherwise every $\cH_x$ has the required natural size.  Let $X_x$
count its members completed through $x$.  Since
$\sum_{x\in X_0}X_x\le\ell|\cZ\cap K_\ell(H)|$, some
$x\in X_0$ has $X_x<c'''_\gamma\mu_\ell(n,q)/n$.
The family $\cH_x$ is determined outside $x$ after the restricted rank
is fixed, so it gives the required fingerprint.  Choose the deficit
constants successively small enough for the lower-tail estimate at
each family-size constant.  Each step records one vertex and one
rank, and there are at most $r-1$ steps.
\end{proof}

\begin{lemma}
\label{lem:balance-fingerprint}
If $\cR(H,q)$ and $\cC(H)$ hold but $\cB(H)$ fails, polynomially
many branches give a sample anomaly or a lower-tail fingerprint.
\end{lemma}

\begin{proof}
Apply Lemma~\ref{lem:median-replacement}, followed by
Lemma~\ref{lem:recursive-fingerprint} with $\ell=r$.
\end{proof}

\subsection{Failure of rooted regularity}

The upper side of $\cR$ is supplied by
Proposition~\ref{prop:uniform-input}(iii).  Its lower side leads to
the same fingerprint without the weight argument.

\begin{lemma}
\label{lem:regularity-fingerprint}
If the lower side of $\cR(H_j,q)$ fails, bounded vertex deletions and
polynomially many restricted-rank choices give a sample anomaly or a
lower-tail fingerprint.
\end{lemma}

\begin{proof}
Choose $x$ and $R\subseteq W\setminus\{x\}$, $|R|\le3r$, with
$d_{K_r}(x;H-R)<c_{\mathrm R}\mu_r(n,q)/n$.
Delete $R\cup\{x\}$ one vertex at a time and record each restricted
rank.  An abnormal degree gives a sample anomaly.  Otherwise every
deletion changes the density by a factor $1+O_r(\nu/n)$; in
particular the final density is at least $p/3$ when the original one
is at least $p/2$.  Write $W'=W\setminus(R\cup\{x\})$ and $h$ for
the final rank.  The graph $H_h(W')$ is fixed by its own birth-time
ordering.  For any eventual root $z$, record its total degree at each
stage still containing $z$.  Lemma~\ref{lem:record-rank-data} makes
the non-$z$ graph at each such stage an outside function.  If $z$
itself is deleted, all subsequent fixed-rank graphs use only non-$z$
birth times.  The remaining exceptional root edges meet the bounded
deleted set; keep their states and all rank and degree identities in
the root event.

If $H_h(W')$ contains fewer than
$c'n^{r-1}q^{\binom{r-1}{2}}$ copies of $K_{r-1}$, apply
Lemma~\ref{lem:recursive-fingerprint} to the family of all
$(r-1)$-sets.  Otherwise these copies form a natural-size outside
family, while the choice of $x$ says that it completes too few of
them.  Taking $c_{\mathrm R}$ sufficiently small gives a lower-tail
fingerprint at $x$.
\end{proof}

\subsection{Median failure and the two-star fingerprint}

Here we retain many light roots so that the probability argument can
later avoid the small set with unusually large coefficient load.
We use the following entropy inequality, with natural logarithms.

\begin{lemma}[Shearer]
\label{lem:shearer}
Let $Z=(Z_i)_{i\in I}$ be a finite-valued random vector, and let
$\mathcal A$ be a family of subsets of $I$ such that every coordinate
lies in at least $t$ members, where $t\ge1$ is an integer.
Writing $Z_A=(Z_i)_{i\in A}$, we have
\[
 tH(Z)\le\sum_{A\in\mathcal A}H(Z_A).
\]
\end{lemma}

We apply this to a uniform factor, using its block at each vertex as
a projection.

\begin{lemma}
\label{lem:median-fingerprint}
Suppose that $\cA(H_j,q)$ and $\cR(H_j,q)$ hold but $\cC(H_j)$
fails.  Then either a sample anomaly occurs, or there is a heavy root
$x$ and a set $\mathcal Y$ of at least $c_{\rm L}|W|$ light roots
such that every $y\in\mathcal Y$ gives a two-star fingerprint after
polynomially many choices of restricted ranks, a dyadic scale and
root degrees.
\end{lemma}

\begin{proof}
Set $m=|W|$ and $d=\mu_r(n,q)/n$.  Choose an $(r-1)$-set $Y$ where
the median condition fails and $x\notin Y$ maximizing
$M=w_H(Y\cup\{x\})$.  Then
$M>n^{-2(r-1)}\Phi(H)$, and at least $(m-r)/2$ vertices
$z\notin R:=Y\cup\{x\}$ satisfy $w_H(Y\cup\{z\})<M/2$.
Call this set $\mathcal L_Y$.

Choose a uniform factor $\mathbf F$ of $H-R$.  Let $K_z$ be its block
through $z$ and $\mathsf h_z=H(K_z)$.  Encoding the factor by its
block indicators, each coordinate belongs to exactly $r$ vertex
projections.  Shearer's inequality and rooted regularity therefore give
\[
 r\log M\le\sum_{z\notin R}\mathsf h_z,
 \qquad \mathsf h_z\le\log(C_rd).
\]
The barrier and the lower bound on $M$ give
$r\log M\ge m\log d-O_r(m)$.  Thus the non-negative deficits
$\log(C_rd)-\mathsf h_z$ have total $O_r(m)$.
For a sufficiently large constant $K_r$, fewer than $m/8$ vertices
have $\mathsf h_z<\log d-K_r$.  Thus
$\mathcal Y=\{y\in\mathcal L_Y:\mathsf h_y\ge\log d-K_r\}$ has
size at least $c_{\rm L}m$.

Fix $y\in\mathcal Y$, put $W_0=W\setminus(Y\cup\{x,y\})$, and
for $A\in\binom{W_0}{r-1}$ set $\alpha_A=\Phi(H[W_0]-A)$.
Let $\mathcal S_y$ contain those $A$ for which $A\cup\{y\}$ is a
clique and $\alpha_A>0$.  The block $K_y$ equals $A\cup\{y\}$ with
probability $p_A=\alpha_A/M$, so $\sum_{A\in\mathcal S_y}\alpha_A=M$.
Moreover $|\mathcal S_y|\le C_rd$ and
\[
 \sum_{A\in\mathcal S_y}p_A\log(p_Ad)
       =\log d-\mathsf h_y\le K_r.
\]
The negative terms have absolute sum at most $C_r$, since
$t\log t\ge-1$ for $t\ge0$.  Hence, for a sufficiently large constant $B$,
\[
 \sum_{\substack{A\in\mathcal S_y\\ \alpha_A>BM/d}}\alpha_A
 \le\frac{K_r+C_r}{\log B}M\le M/5.
\]

Choose a dyadic $m_0$ with $M\le m_0<2M$, and put $b=Bm_0/d$.
Record $h=|E(H[W_0])|$ and define
\[
 c_A=\min\{\Phi(H_h(W_0)-A)/b,1\}
       \ind_{\{H_h(W_0)[A]=K_{r-1}\}},
 \qquad Z(t)=\sum_Ac_A\prod_{u\in A}t_u.
\]
For each fixed $h,m_0$, these coefficients are determined by the
birth times in $W_0$, hence outside both candidate roots.  On the
branch $h=|E(H[W_0])|$, the induced order gives
$H_h(W_0)=H[W_0]$, so the coefficients use precisely the weights
$\alpha_A$ above.  Truncation loses at most $M/5$ of the weight at
$y$, whereas the total untruncated weight at $x$ is at most
$w_H(Y\cup\{y\})<M/2$.  Therefore
\[
 \frac{4M}{5b}\le Z(N_H(y)\cap W_0)\le\frac Mb,
 \qquad Z(N_H(x)\cap W_0)<\frac{M}{2b}.
\]
Since $d/(2B)<M/b\le d/B$, this is a two-star fingerprint.

To use the two stars at exact ranks, also record the state of $xy$
and their relevant degrees.  Delete the bounded set $Y$ successively;
an abnormal degree at any required restricted rank gives a sample
anomaly.  Otherwise the restricted densities remain within a factor
$1+O_r(\nu/n)$ of $q$.  If a later estimate chooses
$z\in\{x,y\}$ as root, all branch conditions involving its star
remain in its event; intersecting with them cannot increase its
conditional probability.  The remaining conditions are outside data,
as in Lemma~\ref{lem:record-rank-data}.
There are polynomially many choices of vertices, ranks and degrees,
and $O(n\log n)$ dyadic values of $m_0$.  The construction works for
every $y\in\mathcal Y$.
\end{proof}

\begin{proof}[Proof of Proposition~\ref{prop:structural-fingerprints}]
The upper side of $\cR$ follows from
Proposition~\ref{prop:uniform-input}(iii), and a lower failure is
handled by Lemma~\ref{lem:regularity-fingerprint}.
When $\cR$ holds but $\cB$ fails, apply
Lemma~\ref{lem:balance-fingerprint} if $\cC$ holds, and
Lemma~\ref{lem:median-fingerprint} otherwise.
\end{proof}
\section{Uniform random-graph input}
\label{sec:upper-regularity}

We collect the random-graph estimates used in the one-set reduction:
edge loads for the deletion martingale, rooted degrees for the
fingerprints, coefficients for Engine~A, and overlaps for Engine~B.
All estimates use the common birth-time coupling.  Throughout,
$b=b(n)=o(n)$ is a deterministic defect budget, and
$\eta_\ell(q):=n^{\ell-2}q^{\binom\ell2-1}$ for $2\le\ell\le r$.
The scales $T(q)$ and $A(q)$ are as in
Section~\ref{sec:deletion-process}; coefficient-regular roots are
defined below.

\begin{proposition}
\label{prop:uniform-input}
There is a decreasing event $\cU$ with $\Pr(\cU)\ge c_r>0$ such that,
for every $W\subseteq[n]$ with $|[n]\setminus W|\le b$, there is an
event $\cO_W$ satisfying
\[
 \Pr(\cO_W^c)\le\exp(-cp|W|^2).
\]
On $\cU\cap\cO_W$, the following conclusions hold simultaneously for
every $W'\subseteq W$ with $|W\setminus W'|\le10r^2$ and every rank
$j\ge p\binom{|W'|}{2}/3$.  Write
 $H:=H_j(W')$ and $q:=j/{\binom{|W'|}{2}}.$

\begin{enumerate}[label=(\roman*),leftmargin=2.7em]
\item \emph{Exact-rank coupling.}
      The graph $H$ is sandwiched between the two birth-time graphs in
      \eqref{eq:rank-sandwich}.  For every root $x\in W'$, the outside
      event $\cV_{W',x,j}\in\mathscr F_{W'}(x)$ from
      Lemma~\ref{lem:ranked-star} holds, and on this event
      \[
       \Pr\left(
       \left|d_H(x)-q(|W'|-1)\right|>\nu q|W'|
       \ \middle|\ \mathscr F_{W'}(x)\right)
       \le \exp(-cq|W'|).
      \]
      Conditional on $\mathscr F_{W'}(x)$ and on $d_H(x)=k$, the
      neighbourhood of $x$ is a uniformly random $k$-subset of
      $W'\setminus\{x\}$.
\item \emph{Edge loads.}
      Every edge of $H$ lies in at most
      $C_r(\eta_r(q)+\log n)$ copies of $K_r$, and
      \[
       \sum_{f\in E(H)}d_{K_r}(f;H)^2
       \exp\!\left(\frac{c_rd_{K_r}(f;H)}{T(q)}\right)
       \le C_rA(q)n^2q\eta_r(q)T(q).
      \]
\item \emph{Rooted clique degrees.}
      For every $2\le\ell\le r$, every vertex lies in at most
      $C_r\mu_\ell(n,q)/n$ copies of $K_\ell$, also after deleting any
      fixed set of at most $3r$ further vertices.
\item \emph{Rooted coefficients.}
      All but at most $c_{\rm L}|W'|/2$ vertices of $W'$ are
      coefficient-regular in $H$ at density $q$.  The same conclusion
      holds after arbitrary pointwise deletion or $[0,1]$-weighting of
      terms in the coefficient sums, including the restrictions used
      in the rainbow exposure.
\item \emph{Janson overlaps.}
      For every lower-tail fingerprint of uniformity
      $2\le\ell\le r$, its overlap parameter is at most
      $C_r(1+\eta_\ell(q))$.
\end{enumerate}
All conclusions remain valid after the bounded vertex deletions and
restricted-rank choices used in Sections~\ref{sec:fingerprints} and
\ref{sec:fixed-rank}.
\end{proposition}

The cutoff $p/3$ leaves room for these restricted ranks.  On
$\mathcal E_W$, the original deletion process stops at a rank of density
at least $p/2$.  As long as no sample anomaly has already occurred,
deleting one vertex changes the density by a factor
$1+O_r(\nu/n)$.  Since only $O_r(1)$ vertices are deleted in forming
one witness, every restricted rank used later is still at least $p/3$
for large $n$.

We first identify the conditional root-neighbourhood law, then build
$\cU$ in the birth-time graphs and transfer its bounds to exact ranks.

\subsection{Exact-rank coupling and root neighbourhoods}

\begin{lemma}
\label{lem:rank-comparison}
For every $W\subseteq[n]$ with $|[n]\setminus W|\le b$, there is an
event $\cO_W$ satisfying
\[
 \Pr(\cO_W^c)\le \exp(-cp|W|^2)
\]
such that the following holds simultaneously for every
$W'\subseteq W$ with $|W\setminus W'|\le10r^2$, every $x\in W'$,
and every rank $j\ge p\binom{|W'|}{2}/3$.  Writing
$q=j/\binom{|W'|}{2}$, one has
\begin{equation}
 G_{(1-\nu)q}[W']\subseteq H_j(W')
 \subseteq G_{(1+\nu)q}[W'],
 \label{eq:rank-sandwich}
\end{equation}
where parameters larger than $1$ are truncated at $1$.  Moreover,
the outside event $\cV_{W',x,j}$ from
Lemma~\ref{lem:ranked-star} holds for every such triple.
\end{lemma}

The following one-root estimate supplies both the outside event and
the conditional neighbourhood law.

\begin{lemma}
\label{lem:ranked-star}
Fix a sufficiently small constant \(\nu=\nu(r)>0\).  For every
\(W\subseteq[n]\) with $|W|=n-o(n)$, \(x\in W\), and rank \(j\) with
\(p/3\le q_j(W)\le1\), there is an event
\(\cV_{W,x,j}\in\mathscr F_W(x)\) such that
\[
 \Pr(\cV_{W,x,j}^c)\le \exp(-c q_j(W)|W|^2)
\]
and, on \(\cV_{W,x,j}\),
\[
 \Pr\left(
 \left|d_{H_j(W)}(x)-q_j(W)(|W|-1)\right|>\nu q_j(W)|W|
 \ \middle|\ \mathscr F_W(x)
 \right)
 \le \exp(-c q_j(W)|W|).
\]
Conditional on \(\mathscr F_W(x)\) and on
\(d_{H_j(W)}(x)=k\), the neighbour set of \(x\) in \(H_j(W)\) is a
uniformly random \(k\)-subset of \(W\setminus\{x\}\).
\end{lemma}

\begin{proof}[Proof of Lemma~\ref{lem:ranked-star}]
Fix $W,x,j$ and put $m=|W|$, $q=q_j(W)$ and $d_x=m-1$.
Write $B_x(t):=|G_t[W\setminus\{x\}]|$.  This is outside-star
measurable and has distribution $\operatorname{Bin}(\binom{m-1}{2},t)$.
Set $t_\pm=(1\pm\nu/4)q$, and define $\cV_{W,x,j}$ by
\[
 B_x(t_-)\le j-(1-\nu/2)qd_x,
 \qquad
 B_x(t_+)\ge j-(1-\nu/2)qd_x,
\]
where the second condition is imposed only if $(1+\nu)q<1$.
Each threshold is separated from the corresponding mean by
$\Theta(\nu qm^2)$, so Chernoff's inequality gives
$\Pr(\cV_{W,x,j}^c)\le e^{-cqm^2}$.

Condition now on the outside birth times, and suppose $\cV_{W,x,j}$
holds.  Let $\tau_{(j)}$ be the $j$-th birth time.  If
$d_{H_j(W)}(x)<(1-\nu)qd_x$, then $\tau_{(j)}\ge t_-$: otherwise
$j=B_x(\tau_{(j)})+d_{H_j(W)}(x)<j-\nu qd_x/2$.
Thus fewer than $(1-\nu)qd_x$ star birth times lie below $t_-$,
a Chernoff deviation with probability at most $e^{-cqd_x}$.

The upper deviation is impossible when $(1+\nu)q\ge1$.
Otherwise, if $d_{H_j(W)}(x)>(1+\nu)qd_x$ and
$\tau_{(j)}\le t_+$, more than $(1+\nu)qd_x$ star birth times lie
below $t_+$.  If $\tau_{(j)}>t_+$, the upper quantile condition
instead forces fewer than $(1-\nu/2)qd_x$ star birth times below
$t_+$.  Each event has conditional probability at most $e^{-cqd_x}$.
This proves the degree bound, after changing $c$.
Finally, permutation symmetry of the star coordinates makes the
neighbourhood uniform conditional on its size and the outside birth
times.
\end{proof}

\begin{proof}[Proof of Lemma~\ref{lem:rank-comparison}]
For a fixed $W',j$, put $q=q_j(W')$.
Chernoff's inequality gives $|G_{(1-\nu)q}[W']|\le j$ and,
when $(1+\nu)q<1$, $|G_{(1+\nu)q}[W']|\ge j$, except with
probability $e^{-cq|W'|^2}$.  These inequalities imply the rank
sandwich.  Lemma~\ref{lem:ranked-star} has the same failure bound
for each outside event $\cV_{W',x,j}$.
There are $n^{O_r(1)}$ admissible triples, all with
$q\ge p/3$ and $|W'|=|W|-O_r(1)$.  Intersecting their outside
events and rank comparisons therefore gives the required $\cO_W$.
\end{proof}

We next record how bounded vertex deletions preserve one-root
measurability.

\begin{lemma}
\label{lem:record-rank-data}
Let $R\subseteq W$ have bounded size and $z\in W\setminus R$.
After fixing an initial rank on $K_W$, the restricted ranks,
intermediate degrees and edge states used in a bounded deletion
sequence can be recorded with $n^{O_r(1)}$ labels.
On each labelled branch, all graph data away from $z$ agree with
$\mathscr F_W(z)$-measurable functions.  The remaining consistency conditions
can be intersected with the root event without increasing its
conditional probability.
\end{lemma}

\begin{proof}
List the deleted vertices as $v_0,\ldots,v_{t-1}$ and put
$W_i:=W\setminus\{v_0,\ldots,v_{i-1}\}$, with $z\in W_i$ throughout.
Record the rank $h_i$ on $K_{W_i}$ and the degree
$k_i:=d_{H_{h_i}(W_i)}(z)$ at every stage.  There are only $O_r(1)$
stages, $O(n^2)$ choices per rank and $n$ per degree, so these records
and the boundedly many exceptional edge states give $n^{O_r(1)}$
labels.

For fixed numerical labels, restriction of the common ordering gives
\begin{equation}
 H_{h_i}(W_i)-z
 =H_{h_i-k_i}(W_i\setminus\{z\}).
 \label{eq:rank-bookkeeping-invariant}
\end{equation}
The right side is determined by $\mathscr F_W(z)$.
If a degree at $u\ne z$ is used, also record the state
$\varepsilon_{i,u}$ of $zu$ in $H_{h_i}(W_i)$.  Then
\[
 d_{H_{h_i}(W_i)}(u)
 =d_{H_{h_i-k_i}(W_i\setminus\{z\})}(u)+\varepsilon_{i,u}.
\]
Thus the next rank is the outside quantity
$h_i-d_{H_{h_i-k_i}(W_i\setminus\{z\})}(v_i)$ minus the recorded
state $\varepsilon_{i,v_i}$.  Induction separates every consistency
condition into an outside condition and a condition on the star with
the outside configuration fixed.

Intersect the event supplied by the one-root estimate with these
conditions.  For each outside configuration, the resulting event is
a subset of the original event, so its conditional probability cannot
increase.  If that estimate holds only on an outside event, include
that event in the intersection as well.  This gives the local event
required by Definition~\ref{def:single-root-witness}.
\end{proof}

The next lemma compares a uniform neighbourhood with independent
selectors.  Its polynomial loss is harmless because the one-root
exponent is much larger than $\log n$.

\begin{lemma}
\label{lem:hypergeom-transfer}
Let $X\subseteq[N]$ be obtained by selecting each element independently
with probability $\rho$, and let $K$ be a uniformly random $k$-subset
of $[N]$, where $\rho=k/N$.  For every event $\mathcal E$ depending
only on the selected set,
\[
 \Pr(K\in\mathcal E)\le(N+1)\Pr(X\in\mathcal E).
\]
\end{lemma}

\begin{proof}
The conditional law of $X$ given $|X|=k$ is uniform over the
$k$-subsets.  Since $k$ is a mode of $\operatorname{Bin}(N,k/N)$,
$\Pr(|X|=k)\ge1/(N+1)$.
\end{proof}

\subsection{Construction of the common event \texorpdfstring{$\cU$}{U}}

The numerical estimates will hold on the following common event.

\begin{lemma}
\label{lem:upper-regularity-summary}
There is a decreasing event
\[
 \cU=\cU_{\rm coeff}\cap\cU_{\rm deg}\cap\cU_{\rm ov}
\]
with $\Pr(\cU)\ge c_r>0$ such that the conclusions of
Lemmas~\ref{lem:coefficient-average-bounds},
\ref{lem:rooted-degree-caps} and
\ref{lem:janson-overlap-caps} hold simultaneously for every
$p/8\le q\le1$ and every induced subgraph of $G_q$.
\end{lemma}

Use the geometric grid $\mathfrak Q=\{q_0,\ldots,q_J\}$, where
$q_i=2^ip/8$ for $i<J$, $q_J=1$, and $q_{J-1}<1\le2q_{J-1}$.
It has $O_r(\log n)$ points.  We prove the estimates on this grid;
Lemma~\ref{lem:fixed-rank-monotonicity} then transfers them to every
$q$ between consecutive points.  Induced subgraphs and pointwise
restrictions only decrease the relevant counts, so they require no
additional union bound.

Here ``decreasing'' refers to adding edges in the coupled graphs.
Such events are increasing in the product birth-time order, since
larger birth times remove edges.  Hence Harris's inequality gives,
for every graph-decreasing event $\mathcal L$,
\[
 \Pr(\mathcal L)
 \le \frac{\Pr(\mathcal L\cap\cU)}{\Pr(\cU)}.
\]
A positive constant lower bound for $\Pr(\cU)$ is therefore enough.

We shall use the following direct reformulation of the fixed-degree
polynomial concentration theorem of Johansson--Kahn--Vu
\cite[Theorem~5.6, Corollary~5.7 and Remark~5.2]{JKV}.  Their parameter
$n$ is the number of Bernoulli variables.  Our graph polynomials use at
most $\binom n2$ edge variables (unused coordinates may be added), so
replacing that parameter by the number of graph vertices only changes
the fixed power in the derivative hypothesis.  In every application
the Bernoulli coordinates have a common success probability; the
deterministic endpoint $q=1$ is handled directly.

\begin{theorem}
\label{thm:polynomial-concentration}
Fix an integer $D\ge1$ and constants $C_0,\alpha,\delta>0$.  Let
$X_1,\ldots,X_N$ be independent $\operatorname{Ber}(q)$ variables with
a common parameter $0<q<1$, where $N\le\binom n2$, and let $f$ be a
multilinear polynomial of degree at most $D$ in these variables, with
all coefficients in $[0,C_0]$.  Write $\EE_L'[f]$ for the
expectation of the non-constant part of the derivative in the
coordinates $L$.  If
\[
 \EE f=\omega(\log n),
 \qquad
 \max_{L\ne\varnothing}\EE_L'[f]
 \le n^{-\alpha}\EE f,
\]
then
\[
 \Pr\bigl(|f-\EE f|>\delta\EE f\bigr)\le n^{-\omega(1)}.
\]
If instead $\EE f\le M$ and
 $M\ge\omega(\log n)+n^\alpha
      \max_{L\ne\varnothing}\EE_L'[f],$
then
\[
 \Pr\bigl(f>(1+\delta)M\bigr)\le n^{-\omega(1)}.
\]
\end{theorem}

In every application below the degree, coefficient bound and derivative
margin are uniform.  The resulting $n^{-\omega(1)}$ estimates are
therefore uniform over the polynomially many choices under
consideration.

\subsubsection{Rooted coefficient and edge-load bounds}

Set $d:=r-2$ and $\beta:=(r+1)/2$.  For distinct vertices $x,u$ and
$0\le k\le d$, define
\begin{equation}
 \Gamma_{x,u,k}(H,q)
 :=q^{d-k}
 \sum_{\substack{B\in\binom{V(H)\setminus\{x,u\}}d\\
                   H[u\cup B]=K_{r-1}}}
 \binom{|B\cap N_H(x)|}{k}.
 \label{eq:Gamma-local}
\end{equation}
Write $D_{xu}(H):=\Gamma_{x,u,d}(H,q)$, which is independent of $q$
and does not use the coordinate $xu$.  When $xu\in E(H)$, it counts
the copies of $K_r$ containing that edge.

The exponential sums below control the average contribution of large
coefficients.  The rough maximum bound is used only to bound the
relative change in the factor count after one edge deletion.

\begin{lemma}
\label{lem:coefficient-average-bounds}
There is a decreasing event $\cU_{\rm coeff}$ with
$\Pr(\cU_{\rm coeff})\ge2/3$ such that, on this event, simultaneously
for every $p/8\le q\le1$,
\begin{align}
 \sum_{xu\in E(G_q)}D_{xu}(G_q)^2
 \exp\!\left(\frac{c_rD_{xu}(G_q)}{T(q)}\right)
 &\le C_rA(q)n^2q\eta_r(q)T(q),\notag\\
 \sum_{x\ne u}\sum_{k=0}^{d}
 \Gamma_{x,u,k}(G_q,q)^2
 \exp\!\left(\frac{c_r\Gamma_{x,u,k}(G_q,q)}{T(q)}\right)
 &\le C_rA(q)n^2\eta_r(q)T(q),
 \label{eq:coefficient-average-binomial}\\
 \max_{x,u,k}\Gamma_{x,u,k}(G_q,q)
 &\le C_r\bigl(\eta_r(q)+\log n\bigr).\notag
\end{align}
The conclusions remain valid in induced subgraphs and after arbitrary
pointwise deletion of summands or multiplication of them by weights in
$[0,1]$, including all colour-restricted and weighted sub-sums used
below.
\end{lemma}

Appendix~\ref{app:coefficient-profile} proves this lemma by a
single-coefficient tail estimate, followed by averaging over the
common grid.  The first sum controls the deletion martingale; the
second supplies regular roots for Engine~A.
For later use, write
\[
 \mathsf S_x(H,q):=
 \sum_{u\ne x}\sum_{k=0}^{r-2}
 \Gamma_{x,u,k}(H,q)^2
 \exp\!\left(\frac{c_r\Gamma_{x,u,k}(H,q)}{T(q)}\right).
\]
We call $x$ \emph{coefficient-regular} if
\[
 \mathsf S_x(H,q)
 \le C_{\rm reg}A(q)n\eta_r(q)T(q),
\]
where $C_{\rm reg}=C_{\rm reg}(r)$ is a sufficiently large constant.
The aggregate bound shows that at most $C_rn/C_{\rm reg}$ roots
fail this condition.  Choose $C_{\rm reg}$ large enough, also for the
constant changes in the rank transfer below, that this number is at
most $c_{\rm L}n/4$.  A regular root can then be chosen among any
$c_{\rm L}n$ light roots.

\subsubsection{Rooted clique-degree bounds}

Recall that $d_{K_\ell}(v;H-R)$ is the number of copies of $K_\ell$ in
$H-R$ containing $v$, where $|R|\le3r$ in our applications.

\begin{lemma}
\label{lem:rooted-degree-caps}
There is a decreasing event $\cU_{\rm deg}$ with
$\Pr(\cU_{\rm deg})=1-o(1)$ such that, on this event, for every
$p/8\le q\le1$, every $2\le\ell\le r$, and every vertex, the number
of copies of $K_\ell$ through that vertex is at most
$C_r\mu_\ell(n,q)/n$, even after deleting a fixed set of at most $3r$
further vertices.  The same conclusion holds in every induced
subgraph.
\end{lemma}

\begin{proof}
At $q=1$ the bound is deterministic.  At a grid point $q<1$, the
clique count through a fixed root has mean
$\Theta_r(n^{\ell-1}q^{\binom\ell2})$.
For $\ell=r$ this is $\omega(\log n)$ by the lower bound on $p$;
for $\ell<r$ it is at least
$n^{(\ell-1)(r-\ell)/r+o(1)}$.
Lemma~\ref{lem:rooted-derivatives} verifies the derivative condition,
so Theorem~\ref{thm:polynomial-concentration} gives the desired cap
with failure probability $n^{-\omega(1)}$.
A union bound over the roots, uniformities and grid points gives a
decreasing event of probability $1-o(1)$.  The common monotonicity
argument extends it to all densities and all vertex deletions.
\end{proof}

\subsubsection{Janson overlap bounds}

For a graph $H$, a set $T\in\binom{V(H)}t$ and
$1\le t\le\ell-2$, set
\[
 d_{\ell-1}(T;H):=
 |\{A\in K_{\ell-1}(H):T\subseteq A\}|.
\]
We dominate this count by omitting the requirement that the edges
inside $T$ are present:
\[
 \widehat d_{\ell-1}(T;H):=
 \left|\left\{
 B\in\binom{V(H)\setminus T}{\ell-1-t}:
 \binom{T\cup B}{2}\setminus\binom T2\subseteq E(H)
 \right\}\right|.
\]
Let $\widehat\Delta_{\ell-1,t}(H)$ be the maximum of this quantity over
all $t$-sets $T$, and define
\[
 \Lambda_\ell(H,q):=
 1+\sum_{t=1}^{\ell-2}q^{\ell-1-t}
 \widehat\Delta_{\ell-1,t}(H).
\]
The quantity $\Lambda_\ell$ dominates the overlap factor in Janson's
inequality.

\begin{lemma}
\label{lem:janson-overlap-caps}
There is a decreasing event $\cU_{\rm ov}$ with
$\Pr(\cU_{\rm ov})=1-o(1)$ such that, on this event, for every
$p/8\le q\le1$ and every $2\le\ell\le r$,
\begin{equation*}
 \Lambda_\ell(G_q,q)\le C_r(1+\eta_\ell(q)).
\end{equation*}
The conclusion remains true, with changed constants, in every induced
subgraph.
\end{lemma}

\begin{proof}
The case $\ell=2$ is immediate, and $q=1$ is deterministic.
Fix a grid point $q<1$, a uniformity $\ell\ge3$ and a $t$-set $T$.
Put $v=\ell-1-t$ and $m=\binom{\ell-1}{2}-\binom t2$.
The mean of $\widehat d_{\ell-1}(T;G_q)$ is
$\mu=\Theta_r(n^vq^m)$.  Lemma~\ref{lem:rooted-derivatives} and the
two parts of Theorem~\ref{thm:polynomial-concentration} give a cap
$C_r(\mu+\log^2n)$, with failure probability $n^{-\omega(1)}$.

If $\mu<\log^2n$, then
$q^v\log^2n\le C_rn^{-v^2/m}(\log n)^{2+2v/m}=o(1)$,
so its contribution to $\Lambda_\ell$ is absorbed by the initial $1$.
Otherwise, the ratio of its contribution to $\eta_\ell(q)$ is at most
a constant times $n^{1-t}q^{1-\binom{t+1}{2}}$.
This equals $1$ for $t=1$.  For $t\ge2$, the lower bound on $q$
and $t\le\ell-2\le r-2$ give
\[
 1-t+\frac2r\left(\binom{t+1}{2}-1\right)
 =(t-1)\left(\frac{t+2}{r}-1\right)\le0.
\]
Thus this ratio is bounded; in the equality case the logarithmic
factor in the lower bound on $p$ makes it tend to zero.

There are only polynomially many choices of $T,\ell,t$ and the grid
point.  Their caps therefore hold on a common decreasing event of
probability $1-o(1)$.  At $q=1$, the bound
$\widehat\Delta_{\ell-1,t}\le n^{\ell-1-t}$ gives the assertion
directly.  Monotonicity extends these bounds to all densities and
induced subgraphs.
\end{proof}

\begin{proof}[Proof of Lemma~\ref{lem:upper-regularity-summary}]
Lemma~\ref{lem:coefficient-average-bounds} gives
$\Pr(\cU_{\rm coeff})\ge2/3$, while
Lemmas~\ref{lem:rooted-degree-caps} and
\ref{lem:janson-overlap-caps} give
$\Pr(\cU_{\rm deg})=\Pr(\cU_{\rm ov})=1-o(1)$.
Their intersection is decreasing and has probability bounded below by
a positive constant.  By construction, it has all the stated
properties simultaneously.
\end{proof}

\subsection{Transfer to exact ranks}

We use the same comparison for the density grid and the rank sandwich.

\begin{lemma}
\label{lem:fixed-rank-monotonicity}
Fix a constant $C_0\ge1$.  Suppose that $H\subseteq H'$ and
$q\le q'\le C_0q$.  Then, after changing constants depending only on
$r$ and $C_0$, every coefficient, edge-load, rooted-degree and overlap
bound at $(H',q')$ implies the corresponding bound at $(H,q)$.
More precisely,
\[
 \Gamma_{x,u,k}(H,q)\le\Gamma_{x,u,k}(H',q'),
 \qquad
 \Lambda_\ell(H,q)\le\Lambda_\ell(H',q'),
\]
and
\[
 A(q')\le C A(q),\qquad T(q')\le C T(q),\qquad
 \eta_\ell(q')\le C\eta_\ell(q).
\]
Consequently, an exponential sum with constant $c$ at $(H',q')$
controls the analogous sum at $(H,q)$ with a smaller constant
$c'=c'(r,C_0)>0$ in the exponential.
\end{lemma}

\begin{proof}
Every count and coefficient increases with the graph and with $q$.
The displayed scale comparisons follow from $q'/q\le C_0$.
If $X\le X'$ and $T(q')\le CT(q)$, then
$e^{c'X/T(q)}\le e^{cX'/T(q')}$ for $c'\le c/C$.
Summing gives all the asserted estimates, with changed constants.
\end{proof}

\begin{proof}[Proof of Proposition~\ref{prop:uniform-input}]
Take $\cU$ from Lemma~\ref{lem:upper-regularity-summary} and
$\cO_W$ from Lemma~\ref{lem:rank-comparison}.  The latter and
Lemma~\ref{lem:ranked-star} give part~(i).
For a fixed $W',j$, set $H=H_j(W')$, $q=j/\binom{|W'|}{2}$ and
$\bar q=\min\{(1+\nu)q,1\}$.  The rank sandwich gives
$H\subseteq G_{\bar q}[W']$, with $q\le\bar q\le(1+\nu)q$.
Lemma~\ref{lem:fixed-rank-monotonicity} transfers the three binomial
bounds to $(H,q)$, proving parts~(ii), (iii) and~(v).

For part~(iv), the transferred sum over roots is at most
$C_rA(q)n^2\eta_r(q)T(q)$.  At most $C_rn/C_{\rm reg}$ roots can
therefore fail the regularity bound.  By the choice of $C_{\rm reg}$
and $|W'|=(1-o(1))n$, this is at most $c_{\rm L}|W'|/2$.
Pointwise restrictions and weights only reduce the sum.

Finally, Lemma~\ref{lem:record-rank-data} handles every bounded
deletion and restricted-rank label.  Exceptional star conditions are
intersected with the root event rather than included in the
conditioning.  The uniform-neighbourhood law is used only after
fixing the outside birth times and the total root degree, as in
Lemma~\ref{lem:ranked-star}.
\end{proof}

We now use these estimates to turn the fingerprints into one-root events.

\section{From rooted fingerprints to single-root witnesses}
\label{sec:fixed-rank}

We turn the three fingerprints of
Proposition~\ref{prop:structural-fingerprints} into events whose cost is paid at one root.

\begin{proposition}
\label{prop:one-step-witness}
At any relevant rank $j$, on $\cU\cap\cO_W$, if
$\cA(H_j,q_j)$ holds and either $\cR(H_j,q_j)$ or $\cB(H_j)$ fails,
then the resulting event is covered by at most $n^{C_r}$ single-root
witnesses.
\end{proposition}

Sample anomalies follow from the ranked-star estimate.  For the other
two fingerprints we need an upper-tail bound from rainbow exposure and
a weighted lower-tail bound from Janson's inequality.

\subsection{Engine A: the two-star fingerprint}
\label{sec:rainbow-engine}

The two-star gap forces an upper deviation at the light root or a lower
deviation at the heavy root.  We estimate these two events separately.

Write $\theta_q:=\mu_r(n,q)/n$.
Fix one branch of Lemma~\ref{lem:median-fingerprint}.  On this branch
there are two roots $x,y$, a vertex set $\Omega$ disjoint from them, a
graph $H_0$ on $\Omega$, and coefficients
$0\le c_A\le1$, determined outside both root stars, such that
\[
 Z(K):=\sum_{A\in\cH}c_A\ind_{\{A\subseteq K\}},
 \qquad
 \cH\subseteq K_{r-1}(H_0).
\]
The two values in the fingerprint are $Z(K_x)$ and $Z(K_y)$, where
$K_x,K_y$ are the two restricted root neighbourhoods.  After the two
root degrees have been recorded, each $K_z$ is a uniform set of its
prescribed size.

\begin{lemma}
\label{lem:split-two-star}
Fix $a,\rho_0>0$ and take $\nu$ sufficiently small in terms of
$r,a,\rho_0$.  In the setting above, suppose that both
restricted root degrees are normal, that the light root $y$ is
coefficient-regular in the common induced graph obtained after deleting
the fixed heavy root and the fixed $(r-1)$-set, and that
$Z(K_y)\ge a\theta_q$ and $Z(K_y)>(1+\rho_0)Z(K_x)$.
Suppose also that the outside graph satisfies
\[
 |K_{r-1}(H_0)|\le C_r\mu_{r-1}(n,q),
 \qquad \Lambda_r(H_0,q)\le C_rT(q).
\]
Then this branch is contained in the union of two candidate root
events: an upper-deviation event at $y$ and a lower-deviation event at
$x$.  After the outside branch data have been fixed, each event has
pointwise conditional probability at most
\[
 \exp\!\left(-c\frac{\mu_r(n,p)}n\right).
\]
where $c=c(r,a,\rho_0)>0$.
The rank, degree and bounded edge-state data have only $n^{O_r(1)}$
possible values.
\end{lemma}

We first prove the upper-deviation estimate.  The main point is to avoid
conditioning on the final coefficient-regularity event, which depends
on the whole root neighbourhood.  Instead, coefficient regularity will
imply a bound which is already measurable before each colour class is
revealed.

\subsubsection{Upper deviations by rainbow exposure}

Let $m=r-1$ and
\[
 Z(X)=\sum_{A\in\cH}c_A\prod_{u\in A}X_u,
 \qquad \cH\subseteq\binom\Omega m,
 \qquad 0\le c_A\le1,
\]
where the $X_u$ are independent $\operatorname{Ber}(\rho)$ variables.
Set $\mu=\EE Z$.  Independently colour $\Omega$ uniformly with colours
$1,\ldots,m$, and let $Z^\chi$ be the sum over the rainbow members of
$\cH$.  For a fixed colouring, expose the colour classes in order $\{1,\ldots,m\}$ and
write
\[
 \EE[Z^\chi\mid\mathcal F_i^\chi]
 -\EE[Z^\chi\mid\mathcal F_{i-1}^\chi]
 =\sum_{u\in\Omega_i^\chi}
   \alpha_{i,u}^\chi(X_u-\rho),
\]
where $\mathcal F_i^\chi$ is generated by the selectors in the first
$i$ classes.  Let $\mathcal R$ be an event such that, for every
colouring $\chi$ and every $i$,
\begin{equation}
 \mathcal R\subseteq
 \left\{
 \rho\sum_{u\in\Omega_i^\chi}(\alpha_{i,u}^\chi)^2
 \exp\!\left(\frac{c_0\alpha_{i,u}^\chi}{T}\right)
 \le C_1A\theta T
 \right\}.
 \label{eq:predictable-cap-event}
\end{equation}

\begin{lemma}
\label{lem:rainbow-exposure}
Suppose $A\ge1$.  For every fixed $a>0$,
\[
 \Pr(Z-\mu\ge a\theta,\ \mathcal R)
 \le \exp\!\left(-c\frac{\theta}{AT}\right),
\]
where $c>0$ depends only on $r,a,c_0,C_1$.
\end{lemma}

\begin{proof}
Fix a colouring $\chi$ and let $v_i$ denote the left side of
\eqref{eq:predictable-cap-event}.  Before exposing class $i$, stop if
adding $v_i$ would make the accumulated sum exceed
$B:=mC_1A\theta T$.  Write $\widetilde D_\chi$ for the sum of the
increments retained before stopping.  For $0<\lambda\le c_0/T$,
independence within a colour class gives
\[
 \log\EE\!\left[
 e^{\lambda\sum_u\alpha_{i,u}^\chi(X_u-\rho)}
 \mid\mathcal F_{i-1}^\chi\right]
 \le \sum_u\rho(e^{\lambda\alpha_{i,u}^\chi}-1-
                 \lambda\alpha_{i,u}^\chi)
 \le \lambda^2v_i/2.
\]
Iterating this estimate with the accumulated conditional bound yields
$\EE e^{\lambda\widetilde D_\chi}\le e^{\lambda^2B/2}$.
On $\mathcal R$ the process never stops, so
$\widetilde D_\chi=Z^\chi-\EE_XZ^\chi$.

A fixed $m$-set is rainbow with probability $\pi=m!/m^m$.  Hence
$\EE_\chi(Z^\chi-\EE_XZ^\chi)=\pi(Z-\mu)$, and Jensen's inequality gives
\[
 \EE\bigl[\ind_{\mathcal R}e^{\lambda\pi(Z-\mu)}\bigr]
 \le \EE_\chi\EE e^{\lambda\widetilde D_\chi}
 \le e^{\lambda^2B/2}.
\]
Markov's inequality with $\lambda=c'/(AT)$, for a sufficiently small
constant $c'>0$, proves the claim.  In particular, we never condition
on the final event $\mathcal R$.
\end{proof}

We now verify the predictable bound for the polynomials which arise in
the two-star fingerprint.

\begin{lemma}
\label{lem:predictable-root-coefficients}
Let $H_0$ be a graph on $\Omega$, let
$\cH\subseteq K_{r-1}(H_0)$, and let $0\le c_A\le1$.  Let
$X\subseteq\Omega$ be chosen by independent selectors of density
$\rho$, where $\rho/q$ is bounded above and below by fixed constants,
and form the graph $H_X$ by adding a root $z$ adjacent exactly to the
vertices of $X$.  If $z$ is coefficient-regular in $H_X$ at density
$q$, then, for every auxiliary colouring and every stage of the rainbow
exposure, the event in \eqref{eq:predictable-cap-event} holds with
\[
 A=A(q),\qquad T=T(q),\qquad
 \theta=\theta_q.
\]
Moreover, the left side of \eqref{eq:predictable-cap-event} is
measurable before the current colour class is exposed.
\end{lemma}

\begin{proof}
Put $d=r-2$.  Fix a colouring and suppose that $h=i-1$ colour classes have already
been exposed.  For $u$ in the next class, the predictable coefficient
is
\[
 \alpha_{i,u}
 =\rho^{d-h}
 \sum_{\substack{A\in\cH\text{ rainbow}\\u\in A}}
 c_A\prod_{v\in A\cap(\Omega_1\cup\cdots\cup\Omega_h)}X_v.
\]
Write $B=A\setminus\{u\}$, so $|B|=d$.  If a term in this sum is
non-zero, its $h$ previously exposed vertices all belong to
$B\cap N_{H_X}(z)$.  Since $c_A\le1$, dropping the colour restrictions
gives
\begin{align*}
 \alpha_{i,u}
 &\le \rho^{d-h}
 \sum_{\substack{B\in\binom{\Omega\setminus\{u\}}d\\
                  H_0[u\cup B]=K_{r-1}}}
 \binom{|B\cap N_{H_X}(z)|}{h}\le C_r\Gamma_{z,u,h}(H_X,q),
\end{align*}
where the last inequality uses $\rho=\Theta(q)$.  After decreasing the
constant in the exponential,
\begin{align*}
 \rho\sum_{u\in\Omega_i}\alpha_{i,u}^2
 \exp\!\left(\frac{c_0\alpha_{i,u}}{T(q)}\right)
 &\le C_rq\sum_{u\ne z}
 \Gamma_{z,u,h}(H_X,q)^2
 \exp\!\left(\frac{c_r\Gamma_{z,u,h}(H_X,q)}{T(q)}\right)\\
 &\le C_rA(q)nq\eta_r(q)T(q)\\
 &\le C_rA(q)\theta_qT(q).
\end{align*}
The middle inequality is coefficient regularity, and the last uses
\eqref{eq:mu-eta-relation}.  The formula for $\alpha_{i,u}$ depends
only on the selectors in the earlier classes, so the bound is
predictable.
\end{proof}

Combining the preceding two lemmas gives the upper-tail estimate used in the two-star argument.

\begin{lemma}
\label{lem:one-root-polynomial}
Fix $p/4\le q\le1$.  Let $H_0$ be a graph on $\Omega$ with
$|\Omega|=n-o(n)$, let $\cH\subseteq K_{r-1}(H_0)$, and let
$0\le c_A\le1$ be fixed by the outside graph.  Let $K$ be a uniformly random $k$-subset of $\Omega$,
where
 $k=(1+O(\nu))q|\Omega|$,
and set
\[
 Z(K):=\sum_{A\in\cH}c_A\ind_{\{A\subseteq K\}},
 \qquad
 \mu_k:=\EE Z(K)
 =\frac{(k)_{r-1}}{(|\Omega|)_{r-1}}\sum_{A\in\cH}c_A.
\]
Suppose $\mu_k\le C_0\theta_q$.  For every fixed $a>0$, there is
$c=c(r,a,C_0)>0$ such that
\begin{equation}
 \Pr\bigl(Z(K)-\mu_k\ge a\theta_q,\
       \text{$K$ makes the root coefficient-regular}\bigr)
 \le \exp\!\left(-c\frac{\mu_r(n,p)}n\right).
 \label{eq:one-root-poly-tail}
\end{equation}
The same bound holds pointwise conditionally at an exact rank, after
all non-root edges have been fixed and the root degree has been
recorded.
\end{lemma}

\begin{proof}
First replace $K$ by independent $\operatorname{Ber}(\rho)$ selectors,
where $\rho=k/|\Omega|=(1+O(\nu))q$.  Let $\mu_\rho$ be the resulting
mean.  Since $r$ is fixed and $k\to\infty$,
\[
 \mu_\rho=(1+O_r(k^{-1}))\mu_k
 =\mu_k+o(\theta_q).
\]
Lemma~\ref{lem:predictable-root-coefficients} shows that the final
coefficient-regularity event is contained in every predictable cap
required by Lemma~\ref{lem:rainbow-exposure}.  Thus an upper deviation of $a\theta_q$ from $\mu_k$ gives one of
$a\theta_q/2$ from $\mu_\rho$, for large $n$.
Applying Lemma~\ref{lem:rainbow-exposure} with $A=A(q)$, $T=T(q)$ and
$\theta=\theta_q$ gives
\[
 \Pr(\text{the event in \eqref{eq:one-root-poly-tail}})
 \le C_r\exp\!\left(-c\frac{\theta_q}{A(q)T(q)}\right)
 \le C_r\exp\!\left(-c\frac{\mu_r(n,p)}n\right)
\]
by \eqref{eq:scale-comparisons}.

Lemma~\ref{lem:hypergeom-transfer} now gives the same estimate for the
uniform $k$-set, at the cost of a factor $|\Omega|+1$.  This factor is
absorbed because $\mu_r(n,p)/(n\log n)\to\infty$.  Finally,
Lemma~\ref{lem:ranked-star} identifies the conditional root
neighbourhood at a fixed rank with a uniform set.  The estimate is
pointwise in the fixed outside graph, so it has the conditional form
required later.
\end{proof}

\subsubsection{Weighted lower deviations}

The following estimate covers both the heavy-root deviation and every
lower-tail fingerprint, including those produced by a link recursion.

\begin{lemma}
\label{lem:weighted-root-lower}
\label{lem:ranked-janson}
Fix $2\le\ell\le r$, $p/4\le q\le1$, a graph $H_0$ on $\Omega$ with
$|\Omega|=n-o(n)$, a family
$\cH\subseteq K_{\ell-1}(H_0)$ and coefficients $0\le c_A\le1$.
Let $K$ be a uniform $k$-subset of $\Omega$, where
$k=(1+O(\nu))q|\Omega|$, and put
$Z(K)=\sum_{A\in\cH}c_A\ind_{\{A\subseteq K\}}$ and
$\mu_k=\EE Z(K)$.  Suppose
$\mu_k\ge c_0\mu_\ell(n,q)/n$ and
$\Lambda_\ell(H_0,q)\le C_r(1+\eta_\ell(q))$.
For every fixed $0<\delta<1$, there is $c=c(r,c_0,C_r,\delta)>0$ such that
\[
 \Pr\bigl(Z(K)\le(1-\delta)\mu_k\bigr)
 \le \exp\!\left(-c\frac{\mu_r(n,p)}n\right).
\]
The bound holds pointwise at an exact rank after the non-root edges
have been fixed and the root degree recorded.
\end{lemma}

\begin{proof}
First use independent selectors of density $\rho=k/|\Omega|=\Theta(q)$.
Write $I_A$ for the indicator that all vertices of $A$ are selected and
$\mu_\rho=\rho^{\ell-1}\sum_Ac_A=(1+o(1))\mu_k$.
For distinct intersecting $A,B$, summing first over $A$ and using
$c_B\le1$ gives
\[
 \Delta_c:=\sum_{A\sim B}c_Ac_B\EE(I_AI_B)
 \le C_r\mu_\rho\sum_{t=1}^{\ell-2}
       q^{\ell-1-t}\widehat\Delta_{\ell-1,t}(H_0)
 \le C_r\mu_\rho(1+\eta_\ell(q)).
\]
The sum is empty when $\ell=2$.
To apply Janson with weights, independently retain each $A$ with
probability $c_A$ and let $X$ count the retained selected sets.
Conditional on the selectors, $\EE X=Z$.  On
$Z\le(1-\delta)\mu_\rho$, Markov's inequality gives
$\Pr(X\le(1-\delta/2)\mu_\rho\mid(I_A)_A)\ge\delta/(2-\delta)$.
Janson's inequality for $X$ therefore bounds the weighted lower tail
by $C_\delta\exp(-c_\delta\mu_\rho^2/(\mu_\rho+\Delta_c))$.
Its exponent is at least a constant times
\[
 \frac{\mu_\ell(n,q)}{n(1+\eta_\ell(q))}
 =\Theta_r\!\left(\min\{\mu_\ell(n,q)/n,nq\}\right)
 \ge c_r\frac{\mu_r(n,p)}n.
\]
For the last inequality, $q\ge p/4$ and
$nq/(\mu_r(n,p)/n)\ge c_r/\eta_r(p)\to\infty$.  Also, writing
$e_j=\binom j2$,
\[
 \frac{\mu_\ell(n,p)}{\mu_r(n,p)}
 =\Theta_{r,\ell}\!\left(
 n^{(r-\ell)(\ell-2)/(r+1)}
 (pn^{2/(r+1)})^{-(e_r-e_\ell)}\right)\ge c_{r,\ell}.
\]
Decreasing $\delta$ absorbs the change of mean.
Lemma~\ref{lem:hypergeom-transfer} transfers the bound to a uniform
$k$-set; its polynomial cost is absorbed since
$\mu_r(n,p)/n\gg\log n$.  Lemma~\ref{lem:ranked-star} gives the
conditional exact-rank version.  Including the recorded degree in
the event can only decrease its conditional probability.
\end{proof}

\subsubsection{Completing the two-star estimate}

\begin{proof}[Proof of Lemma~\ref{lem:split-two-star}]
Let $\mu_x,\mu_y$ be the two hypergeometric means and put
$\bar\mu=(\mu_x+\mu_y)/2$.  Normal root degrees give
$\mu_x=(1+O_r(\nu))\mu_y$, while the outside clique bound gives
$\mu_x+\mu_y\le C_r\theta_q$.
Choose $\nu$ sufficiently small in terms of $a,\rho_0$.
If $\bar\mu<a\theta_q/(1+\rho_0/10)$, then
$Z(K_y)-\mu_y\ge c\theta_q$, so
Lemma~\ref{lem:one-root-polynomial} applies at the regular root $y$.

Otherwise $\bar\mu\ge c\theta_q$.  For a sufficiently small fixed
$c'>0$, the gap $Z(K_y)>(1+\rho_0)Z(K_x)$ forces
\[
 Z(K_y)-\mu_y\ge c'\bar\mu
 \quad\text{or}\quad
 \mu_x-Z(K_x)\ge c'\bar\mu.
\]
Indeed, if both fail, then
$Z(K_y)/Z(K_x)\le(1+O_r(\nu)+c')/(1-O_r(\nu)-c')\le1+\rho_0$.
The first alternative again follows from
Lemma~\ref{lem:one-root-polynomial}.  In the second,
$\mu_x\ge c\theta_q$ and the deficit is a fixed fraction of $\mu_x$,
so Lemma~\ref{lem:weighted-root-lower} applies with $\ell=r$.
Both estimates hold pointwise in the outside branch data.
\end{proof}

\subsection{Engine B: the lower-tail fingerprint}

We now state the conditional estimate for a lower-tail fingerprint.
The recorded degree is kept inside the event; only the family and its
overlap bound are conditions on the outside graph.

\begin{lemma}
\label{lem:lower-tail-fingerprint}
Fix constants $c_0,C_r>0$ and an integer $2\le\ell\le r$.
Let $W\subseteq[n]$ have $|W|=n-o(n)$, let $x\in W$, and let $j$
be a rank with $q=q_j(W)\ge p/3$.
Let $0\le k\le|W|-1$ be a fixed integer with
$k=(1+O(\nu))q(|W|-1)$ and
$0\le j-k\le\binom{|W|-1}{2}$, and put
$H_0=H_{j-k}(W\setminus\{x\})$.
Let $\cH\subseteq K_{\ell-1}(H_0)$ be determined by
$\mathscr F_W(x)$, and let $\mathcal D\in\mathscr F_W(x)$ be an event
on which
\[
 |\cH|\ge c_0n^{\ell-1}q^{\binom{\ell-1}{2}},
 \qquad \Lambda_\ell(H_0,q)\le C_r(1+\eta_\ell(q)).
\]
Write $X_x$ for the number of members of $\cH$ contained in
$N_{H_j(W)}(x)$, and set
\[
 \mathcal E=\mathcal D\cap
 \{d_{H_j(W)}(x)=k,\ X_x<c_1\mu_\ell(n,q)/n\}.
\]
There are $c_1,c>0$, depending only on $r,c_0,C_r$, such that
\[
 \Pr(\mathcal E\mid\mathscr F(x))
 =\Pr(\mathcal E\mid\mathscr F_W(x))
 \le \exp\!\left(-c\frac{\mu_r(n,p)}n\right)
 \quad\text{almost surely}.
\]
Thus $(\mathcal E,x)$ is a single-root witness with parameter $c$, local to $W$.
\end{lemma}

\begin{proof}
Fix an outside configuration in $\mathcal D$ and condition temporarily
on $d_{H_j(W)}(x)=k$.  Lemma~\ref{lem:ranked-star} makes the
neighbourhood a uniform $k$-subset of $W\setminus\{x\}$, while the
outside graph is precisely $H_0$.  Hence its completion count has mean
\[
 \mu_k=\frac{(k)_{\ell-1}}{(|W|-1)_{\ell-1}}|\cH|
 \ge c\mu_\ell(n,q)/n.
\]
Choose $c_1<c/2$ and apply Lemma~\ref{lem:weighted-root-lower} with
all weights equal to one.  Multiplying by the conditional probability
of the recorded degree gives the claimed bound given
$\mathscr F_W(x)$.  Outside $\mathcal D$ the event is empty.
Finally, $\mathcal E\in\mathscr G_W$, so independence gives the same
bound conditional on $\mathscr F(x)$.
\end{proof}

The outside quantile event and any further branch conditions may be
intersected with $\mathcal E$.  Lemma~\ref{lem:record-rank-data}
therefore gives the same estimate after the bounded link recursions.

\subsection{Proof of the one-step reduction and the one-root theorem}

\begin{proof}[Proof of Proposition~\ref{prop:one-step-witness}]
Apply Proposition~\ref{prop:structural-fingerprints} and fix one branch
label.  For each output below, let $\mathcal E$ be the intersection
of the stated local conditions.  We verify its conditional bound for
every outside configuration.  The events $\cU\cap\cO_W$ guarantee
these local conditions, but are not included in $\mathcal E$.

\emph{A sample anomaly.}
Let $x$ be the abnormal root in an induced exact-rank graph on $W'$.
Include the outside quantile event from Lemma~\ref{lem:ranked-star}
and the abnormal degree in $\mathcal E$.
The quantile event is measurable outside $x$, and the lemma gives
conditional probability at most $e^{-cq|W'|}$.
This suffices because $q|W'|/(\mu_r(n,p)/n)\ge c_r/\eta_r(p)\to\infty$.

\emph{A lower-tail fingerprint.}
Fix its root $x$ and recorded degree $k$.
By \eqref{eq:rank-bookkeeping-invariant}, deleting $x$ leaves the graph
at outside rank $j-k$.  Include in $\mathcal E$ the family-size and
overlap conditions, the outside quantile event, the degree identity
and the lower deviation.  The first three conditions are measurable
outside $x$; there Lemma~\ref{lem:lower-tail-fingerprint} gives the required
bound for the last two.
A bounded link recursion adds only $n^{O_r(1)}$ labels by
Lemma~\ref{lem:record-rank-data}.

\emph{A two-star fingerprint.}
Fix the heavy root $x$ and the deleted $(r-1)$-set.
The structural proposition gives at least $c_{\rm L}|W|$ light roots
in the same induced graph.  By
Proposition~\ref{prop:uniform-input}(iv), fewer than half of these
can fail coefficient regularity.  Choose a regular light root $y$,
and branch over its choice and the rank, degree and weight-scale
labels from Lemma~\ref{lem:median-fingerprint}.
The outside clique and overlap bounds required by
Lemma~\ref{lem:split-two-star} follow from
Proposition~\ref{prop:uniform-input}(iii),(v).  Include them and the
relevant outside quantile conditions in each event below.

For the upper alternative rooted at $y$, also include the recorded
$y$-degree, coefficient regularity and the upper deviation.
The coefficient list is determined by the outside graph and the
recorded weight scale; its entries require no extra labels.
Once the branch label is fixed, it and all $x$-star data except $xy$
are measurable outside $y$.
Lemma~\ref{lem:one-root-polynomial} bounds the intersection of
regularity and the deviation, without conditioning on regularity.
For the lower alternative rooted at $x$, include
$\mu_x\ge c\theta_q$, the recorded $x$-degree and the lower deviation.  The coefficient list, the mean condition and all
$y$-star data except $xy$ are measurable outside $x$, and
Lemma~\ref{lem:weighted-root-lower} gives the bound.
The edge $xy$ belongs to the chosen root star; its recorded state,
and all other bounded rank and edge-state conditions, are handled by
Lemma~\ref{lem:record-rank-data}.

In each case $\mathcal E\in\mathscr G_W$.  At an outside configuration
violating the stated outside conditions the event is empty; at every
other configuration the root estimate gives
$\Pr(\mathcal E\mid\mathscr F_W(z))\le e^{-c_r\mu_r(n,p)/n}$,
where $z$ is the chosen root.
Independence lifts this bound to $\mathscr F(z)$.
Thus $(\mathcal E,z)$ is a single-root witness, and the number of
labels is $n^{O_r(1)}$.
\end{proof}

\begin{proof}[Proof of Theorem~\ref{prop:one-root-witness}]
Fix $W$ and work on $\cU$.  The event $\cO_W^c$ is included in the
global event $\cM_W$.  Outside $\cM_W$,
Proposition~\ref{prop:stopping-reduction} says that factor failure
forces a rank at which $\cA$ holds and $\cR$ or $\cB$ fails.
Proposition~\ref{prop:one-step-witness} covers the resulting event
by polynomially many single-root witnesses.

A label records the stopping rank, the bounded list of vertices used
in the link recursion, the median set, the weight scale, the
restricted ranks, the normal root degrees and the states of
$O_r(1)$ exceptional root edges.  There are only $n^{O_r(1)}$
possibilities.  Proposition~\ref{prop:one-step-witness} gives the required
conditional bound for each label, proving \eqref{eq:witness-cover}.
\end{proof}

\section{Iteration of the one-root theorem}
\label{sec:factor-tail-proof}

The pointwise conditional bound allows us to multiply the costs of
successive witnesses, provided later vertex sets exclude earlier
roots.  We carry out this iteration once, allowing an event on the
final remainder to stay in the intersection.  Taking that event to be
the whole space proves Theorem~\ref{thm:main}; taking it to be
clique-free gives the maximal-matching application.

Write $\theta=\mu_r(n,p)/n$ and
\[
 \mathcal T_s:=\{0\le t\le s:n-t\in r\mathbb Z\},
 \qquad
 \mathfrak W_s:=\{W\subseteq[n]:|[n]\setminus W|\le s,\ r\mid|W|\}.
\]
For a defect bound $s$, we may use any fixed budget $b\ge s$ with
$b=o(n)$ in the one-root theorem.

\subsection{Iteration and the proof of Theorem~\ref{thm:main}}
\label{sec:final-iteration}

\begin{lemma}
\label{lem:chain-multiplication}
Let $(\mathcal E_j,x_j)$, $0\le j\le s$, be single-root witnesses
with a common parameter $c>0$, local to $W_j$, where
$x_i\notin W_j$ for $i<j$.
If $\mathcal F$ is measurable outside every $x_j$, then
\[
 \Pr\left(\mathcal F\cap\bigcap_{j=0}^s\mathcal E_j\right)
 \le e^{-c(s+1)\theta}\Pr(\mathcal F).
\]
\end{lemma}

\begin{proof}
Every later event $\mathcal E_i$, $i>j$, is measurable outside $x_j$
by locality.  Conditional expectation with respect to
$\mathscr F(x_0)$ therefore bounds the probability by
$e^{-c\theta}\Pr(\mathcal F\cap\bigcap_{j=1}^s\mathcal E_j)$.
Repeating at $x_1,\ldots,x_s$ proves the claim.
\end{proof}

The next statement includes both applications of the iteration.

\begin{proposition}
\label{prop:terminal-iteration}
Let $r,p,s$ satisfy the assumptions of Theorem~\ref{thm:main}, and
let $\mathcal K$ be a decreasing event.  For each $(n-s-1)$-set $R$, let
$\mathcal F_R\in\mathscr G_R$ satisfy $\mathcal K\subseteq\mathcal F_R$.
Then
\[
 \Pr\bigl(\phi_r^s(G_p)=0,\ \mathcal K\bigr)
 \le e^{-c(s+1)\theta}\max_R\Pr(\mathcal F_R)+e^{-c\mu_r}.
\]
\end{proposition}

\begin{proof}
Let $\mathcal L_s=\{\phi_r^s(G_p)=0\}$ and first work on $\cU$.
The one-root theorem and a union bound give
\[
 \Pr\left(\bigcup_{W\in\mathfrak W_s}\cM_W\right)
 \le\exp\bigl(O((s+1)\log n)-c\mu_r\bigr)
 \le e^{-c'\mu_r},
\]
since $s=o(n)$ and $\theta\gg\log n$.

Outside these global errors, we extract $s+1$ witnesses.
On $\mathcal L_s$, every $G_p-S$ with $|S|\in\mathcal T_s$ lacks a
$K_r$-factor: otherwise deleting $(s-|S|)/r$ blocks of such a factor
would give a matching covering $n-s$ vertices.
Suppose roots $x_0,\ldots,x_{j-1}$ have been chosen.  Let $t_j$ be
the smallest member of $\mathcal T_s$ at least $j$, and complete
these roots to a $t_j$-set $S_j$ using a fixed vertex order.
Then $W_j=[n]\setminus S_j$ belongs to $\mathfrak W_s$ and lacks a
factor.  Theorem~\ref{prop:one-root-witness} supplies a witness
$(\mathcal E_j,x_j)$ with $x_j\in W_j$.
Continue through $j=s$.  The roots are distinct, and each $W_j$
excludes all earlier roots, as required by
Lemma~\ref{lem:chain-multiplication}.

There are at most $n^C$ witness labels per stage, including the root.
Once the earlier roots are fixed, $S_j$ is deterministic, so there are
at most $\exp(O((s+1)\log n))$ labelled chains.
For each fixed chain put $R=[n]\setminus\{x_0,\ldots,x_s\}$.
The event $\mathcal F_R$ is measurable outside every root.  The
preceding lemma therefore bounds the intersection of the chain and
$\mathcal K$ by $e^{-c(s+1)\theta}\Pr(\mathcal F_R)$.
The label count is absorbed because $\theta\gg\log n$.
Adding the global errors proves the asserted bound on $\cU$.
Finally, $\mathcal L_s\cap\mathcal K$ and $\cU$ are decreasing, so
Harris's inequality removes $\cU$ at the constant cost
$1/\Pr(\cU)$.  This cost is absorbed by decreasing $c$.
\end{proof}

\begin{proof}[Proof of Theorem~\ref{thm:main}]
Proposition~\ref{prop:isolated-tail} gives the lower bound.
For the upper bound take $\mathcal K$ and every $\mathcal F_R$ to be
the whole space in Proposition~\ref{prop:terminal-iteration}.
Since $s=o(n)$, its global-error term is absorbed by
$\exp(-c(s+1)\mu_r/n)$.  The perfect-factor corollary follows by
setting $s=0$.
\end{proof}

\subsection{Clique-free refinement and the maximal-matching application}

Retaining the absence of the next larger clique gives an extra
probability factor.  This factor will cancel the cost of choosing a
maximal matching.

\begin{lemma}
\label{lem:next-clique-free}
Fix $k\ge3$.  If $q\ll n^{-2/k}$ and $N=n-o(n)$, then, uniformly in
$N$,
\[
 \Pr\bigl(G(N,q)\text{ is }K_k\text{-free}\bigr)
 =\exp\!\left(-(1+o(1))\binom Nkq^{\binom k2}\right).
\]
\end{lemma}

\begin{proof}
Put $\lambda_N=\binom Nkq^{\binom k2}$.
Harris's inequality gives the lower bound
$(1-q^{\binom k2})^{\binom Nk}=e^{-(1+o(1))\lambda_N}$.
For the upper bound, the Janson dependency sum satisfies
\[
 \frac{\Delta}{\lambda_N}
 \le C_k\sum_{h=2}^{k-1}N^{k-h}q^{\binom k2-\binom h2}
 \le C_k\sum_{h=2}^{k-1}
 n^{-(k-h)(h-1)/k}(qn^{2/k})^{\binom k2-\binom h2}
 =o(1).
\]
Janson's inequality now gives the matching upper bound.
\end{proof}

\begin{proposition}
\label{prop:clique-free-refinement}
Under the assumptions of Theorem~\ref{thm:main}, put
$\lambda=\binom n{r+1}p^{\binom{r+1}{2}}$.  Then
\[
 \Pr\bigl(\phi_r^s(G_p)=0,\ G_p\text{ is }K_{r+1}\text{-free}\bigr)
 \le \exp\!\left(-c(s+1)\frac{\mu_r}n-(1-o(1))\lambda\right)
       +e^{-c\mu_r}.
\]
\end{proposition}

\begin{proof}
Apply Proposition~\ref{prop:terminal-iteration} with $\mathcal K$
the event that $G_p$ is $K_{r+1}$-free and $\mathcal F_R$ the same
event restricted to $R$.
As $|R|=n-s-1=n-o(n)$, Lemma~\ref{lem:next-clique-free} gives
$\Pr(\mathcal F_R)=e^{-(1+o(1))\lambda}$ uniformly.
\end{proof}

We apply this refinement to the remainder of a maximal matching.
The argument also gives a second-moment bound needed later for the
chromatic number.

\begin{proof}[Proof of Theorem~\ref{thm:maximum-clique-remainder}]
Write $\lambda=\mu_k(n,q)$ and $\delta=\mu_{k-1}(n,q)/n$.
For a $K_k$-matching $\mathcal M$ in a graph $F$, let
$\rho(F,\mathcal M)$ be the minimum number of vertices left uncovered
by a $K_{k-1}$-matching in $F-V(\mathcal M)$.
Let $\widehat\rho_{k-1}(F)$ be the maximum of this defect over all
maximal $K_k$-matchings $\mathcal M$.

Fix a deterministic matching $\mathcal M$ of size $m$ and put
$N=n-km$.  If it is maximal in $G(n,q)$, its remainder is $K_k$-free.
If also $\rho(G,\mathcal M)>x$, take the largest integer $s\le x$
with $N-s\in(k-1)\mathbb Z$.  For $x\ge2(k-1)$, we have $s\ge x/2$
and the remainder has no $K_{k-1}$-matching covering $N-s$ vertices.
The edges inside the blocks of $\mathcal M$ are independent of the
remainder.  Thus, for $N=n-o(n)$ and $x=o(n)$,
Proposition~\ref{prop:clique-free-refinement} gives
\begin{align*}
 &\Pr\bigl(\mathcal M\subseteq G,\ \mathcal M\text{ is maximal},\
                   \rho(G,\mathcal M)>x\bigr)\\
 &\qquad\le q^{\binom k2m}
 \left\{\exp(-cx\delta-(1-o(1))\lambda_N)
             +\exp(-cn\delta)\right\},
\end{align*}
where $\lambda_N=\binom Nkq^{\binom k2}$.
Here the assumptions and the scale comparisons are uniform in
$N=n-o(n)$: replacing $n$ by $N$ changes each relevant power of $n$
and $\log n$ by a factor $1+o(1)$.

The number of deterministic matchings with $m$ blocks is at most
$\binom nk^m/m!$.  Fix a large constant $A_0$.
Since $\lambda=o(n)$, for $m\le A_0\lambda$ we have
$N=n-o(n)$ and $\lambda_N=(1-o(1))\lambda$, uniformly.
Summing over these matchings and using
$\sum_m\lambda^m/m!\le e^\lambda$ gives
\[
 \Pr\bigl(\exists\text{ a maximal }\mathcal M:
       |\mathcal M|\le A_0\lambda,\ \rho(G,\mathcal M)>x\bigr)
 \le \exp(-cx\delta+o(\lambda))
       +\exp(\lambda-cn\delta).
\]
The $K_k$-free factor cancels the leading $e^\lambda$ cost.
This cancellation is needed when $k=4$.
Stirling's formula also gives
$\Pr(\nu_k(G)>A_0\lambda)\le\sum_{m>A_0\lambda}\lambda^m/m!
\le e^{-c\lambda}$.

We finish by choosing $x$; the same choice gives the stronger moment
bound.  The lower assumption on $q$ implies $\delta\gg\log n$ and
$\lambda\gg(\log n)^{k/(k-2)}\gg\log n$.
Moreover,
\[
 \frac{\lambda}{n\delta}=\Theta_k(nq^{k-1})=o(1),
 \qquad
 \frac{\delta^2}{\lambda}
 =\Theta_k\!\left((nq^{(k-1)/2})^{k-4}\right)\ge c_k.
\]
In particular $\delta\sqrt\lambda\ge c_k\lambda$.
Apply the refinement with the common defect budget
$b=\lceil\sqrt\lambda\rceil$, so its errors are uniform for
$x\le\sqrt\lambda$.  Combining the preceding bounds, there is a deterministic
$a_n=o(\lambda)$ such that, uniformly for $2(k-1)\le x\le\sqrt\lambda$,
\[
 \Pr(\widehat\rho_{k-1}(G)>x)
 \le e^{a_n-cx\delta}+e^{\lambda-cn\delta}+e^{-c\lambda}.
\]
Choose $\zeta_n\downarrow0$ slowly enough that
$\zeta_n\sqrt\lambda\to\infty$ and
$c\zeta_n\delta\sqrt\lambda-a_n\gg\log n$, and set
$x=\lceil\zeta_n\sqrt\lambda\rceil$.
All three terms are $o(n^{-3})$.
Since $0\le\widehat\rho_{k-1}\le n$, we obtain
\begin{equation}
 \EE\widehat\rho_{k-1}(G(n,q))^2
 \le x^2+n^2\Pr(\widehat\rho_{k-1}>x)
 =o(\lambda).
 \label{eq:maximal-defect-L2}
\end{equation}
In particular, with high probability every maximal $K_k$-matching
has defect at most $x=o(\sqrt\lambda)$, as asserted.
\end{proof}

The remaining sections prove the independent fluctuation theorem for
maximum clique matchings and apply it to the chromatic number.
\section{A central limit theorem for sparse clique matchings}
\label{sec:clique-matching-clt}

We prove Theorem~\ref{thm:matching-clt-intro} by comparing the maximum
clique-matching number with the total clique count.  Although their
expectations can differ substantially, the centred difference has
negligible variance.

We use the following form of the Efron--Stein inequality.

\begin{lemma}[Efron--Stein inequality]
\label{lem:efron-stein}
Let $\xi_1,\ldots,\xi_m$ be independent, let
$Z=f(\xi_1,\ldots,\xi_m)$ be square-integrable, and write
$\mathcal F_i:=\sigma(\xi_j:j\ne i)$.  For any square-integrable
$\mathcal F_i$-measurable random variables $Z_i$, we have
\[
 \Var Z
 \le\sum_{i=1}^m\E\bigl[\Var(Z\mid\mathcal F_i)\bigr]
 \le\sum_{i=1}^m\E(Z-Z_i)^2.
\]
\end{lemma}

The main estimate compares the clique count with the maximum matching.

\begin{lemma}
\label{lem:collision-loss}
Fix $k\ge2$, let $n^{-2/(k-1)}\ll q\ll n^{-2/k}$, and let
$H\sim G(n,q)$.  Write $X$ for the number of copies of $K_k$ in $H$,
$\lambda:=\mu_k(n,q)$, $\nu:=\nu_k(H)$, and $D:=X-\nu$.  Then
\[
 \Var X=(1+o(1))\lambda,
 \qquad
 \Var D=o(\lambda).
\]
\end{lemma}

\begin{proof}
Set $e:=\binom k2$.  For $1\le h<k$, let $I_h$ count unordered pairs
of distinct copies of $K_k$ meeting in exactly $h$ vertices.  Counting
their vertices and edges gives
\[
 \frac{\E I_h}{\lambda}
 =O_k\!\left(n^{k-h}q^{e-\binom h2}\right)
 =O_k\!\left(n^{-(h-1)(k-h)/k}
       (qn^{2/k})^{e-\binom h2}\right)=o(1).
\]
Copies meeting in at most one vertex use disjoint edge coordinates, so
$\Var X=\lambda(1-q^e)+O_k(\sum_{h=2}^{k-1}\E I_h)
=(1+o(1))\lambda$.

To bound $\Var D$, form the conflict graph $\Gamma$ whose vertices are
the copies of $K_k$ in $H$, with two adjacent when they share an
original vertex.  Then $\nu=\alpha(\Gamma)$ and $D=\tau(\Gamma)$,
where $\alpha$ and $\tau$ are the independence and vertex-cover numbers.
For $f\in E(K_n)$, let $D_f$ be the corresponding loss after deleting
$f$, and let $b_f$ count the non-isolated clique-vertices containing
$f$.  Deleting $f$ removes precisely the clique-vertices containing it.
A minimum vertex cover of the remaining graph, together with the
removed non-isolated vertices, covers $\Gamma$.  Hence
$0\le D-D_f\le b_f$.

Since $D_f$ does not depend on the coordinate of $f$,
Lemma~\ref{lem:efron-stein} gives
\[
 \Var D\le\sum_f\E(D-D_f)^2\le\E\sum_f b_f^2.
\]
Let $B$ count the non-isolated vertices of $\Gamma$.  Each belongs to a
conflict edge, so $B\le2\sum_{h=1}^{k-1}I_h$.  Each contributes $e$ to
$\sum_f b_f$, while a pair meeting in $h\ge2$ vertices contributes
$\binom h2$ to $\sum_f\binom{b_f}{2}$.  Thus
\[
 \sum_f b_f^2
 =eB+2\sum_{h=2}^{k-1}\binom h2 I_h
 \le C_k\sum_{h=1}^{k-1}I_h.
\]
Taking expectations proves the claim, including the case $k=2$.
\end{proof}

We next record the two weak-convergence results used to transfer this
comparison to a central limit theorem.

\begin{theorem}[Clique-count central limit theorem {\cite{Ruc88}}]
\label{thm:clique-count-clt}
Fix $k\ge2$, let $n^{-2/(k-1)}\ll q\ll n^{-2/k}$, and let $X$
count the copies of $K_k$ in $G(n,q)$.  With $\lambda:=\mu_k(n,q)$,
\[
 \frac{X-\lambda}{\sqrt\lambda}
 \xrightarrow{\mathrm d}\mathcal N(0,1).
\]
\end{theorem}

\begin{proof}
For $k\ge3$, Ruci\'nski's theorem applies because
$m(K_k):=\max_{F\subseteq K_k,\,e(F)>0}e(F)/v(F)=(k-1)/2$,
$nq^{m(K_k)}\to\infty$, and $n^2(1-q)\to\infty$.
Lemma~\ref{lem:collision-loss} gives $\Var X\sim\lambda$.
For $k=2$, the conclusion is the binomial central limit theorem,
since $\lambda\to\infty$ and $q=o(1)$.
\end{proof}

\begin{lemma}[Slutsky's theorem {\cite[Lemma~2.8]{vanDerVaart98}}]
\label{lem:slutsky}
Let $U_n,V_n$ be real-valued random variables.  If
$U_n\xrightarrow{\mathrm d}U$ and $V_n\xrightarrow{\mathrm P}0$, then
\[
 U_n+V_n\xrightarrow{\mathrm d}U.
\]
\end{lemma}

The matching-number limit now follows directly.

\begin{proof}[Proof of Theorem~\ref{thm:matching-clt-intro}]
Retain the notation of Lemma~\ref{lem:collision-loss}, which gives
$\|D-\E D\|_2=o(\sqrt\lambda)$.  The identity
$\nu-\E\nu=(X-\lambda)-(D-\E D)$, together with
Theorem~\ref{thm:clique-count-clt} and Lemma~\ref{lem:slutsky}, yields
the asserted central limit theorem.  Finally, the reverse triangle inequality gives
$|\sqrt{\Var\nu}-\sqrt{\Var X}|\le\sqrt{\Var D}=o(\sqrt\lambda)$,
and hence $\Var\nu\sim\lambda$.
\end{proof}
\section{Gaussian fluctuations of the chromatic number}
\label{sec:chromatic}

A proper colouring of $\overline H$ is a partition of $V(H)$ into
cliques of $H$.  We first use this observation to approximate
$\chi(\overline H)$ in $L^2$, then combine the approximation with the
clique-count comparison from the preceding section.

For a $K_{r+1}$-matching $\mathcal M$ in $H$, let
$\rho_r(H,\mathcal M)$ be the number of vertices left uncovered by a
maximum $K_r$-matching in $H-V(\mathcal M)$.  Let $Z_{r+2}(H)$ count
the copies of $K_{r+2}$ in $H$.

\begin{proposition}
\label{prop:clique-matching-reduction}
Let $r\ge2$, let $H$ have $n$ vertices, and let $\mathcal M$ be a
maximum $K_{r+1}$-matching in $H$.  Then
\[
 \frac{n-\nu_{r+1}(H)-Z_{r+2}(H)}r
 \le \chi(\overline H)
 \le \frac{n-\nu_{r+1}(H)}r
      +\frac{r-1}{r}\rho_r(H,\mathcal M).
\]
\end{proposition}

\begin{proof}
For the upper bound, use the copies in $\mathcal M$ as colour classes,
then a maximum $K_r$-matching in the remainder, and finally singleton
classes for its uncovered vertices.  Writing $\nu:=\nu_{r+1}(H)$ and
$\rho:=\rho_r(H,\mathcal M)$, this uses
$\nu+(n-(r+1)\nu-\rho)/r+\rho$ classes.

For the lower bound, consider a partition into $q$ cliques of orders
$a_1,\ldots,a_q$, and set $E:=\sum_i(a_i-r)_+$.  At most $\nu$
classes have order at least $r+1$, since each contains a $K_{r+1}$
and these copies are vertex-disjoint.  Any further excess is at most
$Z_{r+2}(H)$: for $a\ge r+2$, one has
$a-r-1\le\binom a{r+2}$, and the copies counted in distinct classes
are distinct.  Thus $E\le\nu+Z_{r+2}(H)$, and the claim follows from
$n\le rq+E$.
\end{proof}

We next show that both errors in this sandwich are negligible.  The
maximal-remainder estimate handles $r\ge3$; for $r=2$, we use the
following triangle-remainder theorem.

\begin{theorem}[Yan~\cite{Yan24}]
\label{thm:Yan-triangle-structure}
Let $n^{-1}\log n\ll p\ll n^{-2/3}$ and $H\sim G(n,p)$.
Choose a maximum triangle matching $\mathcal M$ by a fixed
deterministic rule.  Then
\[
 \Pr\bigl(\rho_2(H,\mathcal M)\le1\bigr)
 \ge1-n^{-\omega(1)}.
\]
Thus $H-V(\mathcal M)$ has a matching leaving at most one vertex
uncovered with this probability.
\end{theorem}

\begin{proof}
The result is stated with high probability in~\cite{Yan24}, but its
proof gives the error above.  Lemmas~3.1 and~3.3 there place $H$ in a
deterministic good family with probability $1-n^{-\omega(1)}$, and
Theorem~4.1 gives the required matching for every graph in this family.
\end{proof}

We can now treat all clique windows together.

\begin{lemma}
\label{lem:chromatic-L2-structure}
Fix $r\ge2$, let
$n^{-2/r}(\log n)^{1/\binom r2}\ll p\ll n^{-2/(r+1)}$, and let
$H\sim G(n,p)$.  With $\lambda:=\mu_{r+1}(n,p)$, we have
\[
 \E\left[\left(\chi(\overline H)
                  -\frac{n-\nu_{r+1}(H)}r\right)^2\right]
 =o(\lambda).
\]
For $r=2$, the left-hand side is in fact $O(1)$.
\end{lemma}

\begin{proof}
Choose a maximum $K_{r+1}$-matching $\mathcal M$ by a fixed
deterministic rule, and write $\rho:=\rho_r(H,\mathcal M)$.
For $r\ge3$, the second-moment estimate~\eqref{eq:maximal-defect-L2}
gives $\E\rho^2=o(\lambda)$, since $\mathcal M$ is also maximal.
For $r=2$, Theorem~\ref{thm:Yan-triangle-structure} and the bound
$\rho\le n$ give $\E\rho^2\le1+n^2n^{-\omega(1)}=O(1)=o(\lambda)$,
since $\lambda\gg(\log n)^3$.

It remains to bound $Z:=Z_{r+2}(H)$, uniformly for $r\ge2$.  Put
$j:=r+2$ and $b:=pn^{2/(j-1)}=o(1)$.  Then
$\E Z=\Theta_j(b^{\binom j2})=o(1)$.  The contribution to
$\E[Z(Z-1)]$ from pairs meeting in $h$ vertices is
\[
 O_j\!\left(n^{2j-h}p^{2\binom j2-\binom h2}\right)
 =O_j\!\left(n^{h(h-j)/(j-1)}
                  b^{2\binom j2-\binom h2}\right)=o(1)
 \qquad(0\le h<j).
\]
Thus $\E Z^2=o(1)$.  By
Proposition~\ref{prop:clique-matching-reduction}, the error
$R:=\chi(\overline H)-(n-\nu_{r+1}(H))/r$ lies between $-Z/r$ and
$(r-1)\rho/r$.  Hence $\E R^2\le C_r(\E Z^2+\E\rho^2)=o(\lambda)$,
with the stated $O(1)$ bound when $r=2$.
\end{proof}

We can now identify the centred fluctuations and derive all three
conclusions at once.

\begin{proof}[Proof of Theorem~\ref{thm:chromatic-clt} and
Corollary~\ref{cor:chromatic-count-equivalence}]
Let $H\sim G(n,p)$, and write $Y:=\chi(\overline H)$,
$X:=X_{r+1}(H)$, $\lambda:=\mu_{r+1}(n,p)$, and
$\nu:=\nu_{r+1}(H)$.  Set $D:=X-\nu$ and
$R:=Y-(n-\nu)/r$.  Lemmas~\ref{lem:collision-loss}
and~\ref{lem:chromatic-L2-structure} give
$\Var D=o(\lambda)$ and $\Var R=o(\lambda)$.  Since $\E X=\lambda$,
\[
 r(Y-\E Y)+(X-\lambda)
 =(D-\E D)+r(R-\E R)
 =o_{L^2}(\sqrt\lambda).
\]
This proves the centred $L^2$ approximation in the corollary.
Theorem~\ref{thm:clique-count-clt}, with $k=r+1$, and
Lemma~\ref{lem:slutsky} give the asserted Gaussian limit for $Y$.
The reverse triangle inequality also gives
$|r\sqrt{\Var Y}-\sqrt{\Var X}|=o(\sqrt\lambda)$; since
$\Var X\sim\lambda$, it follows that $\Var Y\sim\lambda/r^2$.
Finally, expanding the squared norm of the displayed approximation
yields $r\operatorname{Cov}(Y,X)=-(1+o(1))\lambda$.  Dividing by
the standard deviations proves $\operatorname{Corr}(Y,X)\to-1$.
\end{proof}
\section{Concluding remarks}\label{sec:remark}

In this paper, we determine the exponential order of almost-factor
failure throughout the full sparse clique window, prove the corresponding
chromatic central limit theorem, and identify the centred count of maximum
independent sets as its complete first-order fluctuation source.  We finish
with three questions which are not resolved by the present argument.

\subsection*{The exact logarithmic rate}

Theorem~\ref{thm:main} determines the exponent up to constants, whereas
the local obstruction is known more precisely.  Proposition
~\ref{prop:isolated-tail} gives
\[
 -\log\Pr(Y_r\ge s+1)
 \sim(s+1)\binom{n-1}{r-1}p^{\binom r2}.
\]
It is natural to expect that the same leading term governs the whole
failure event:
\[
 -\log\Pr\bigl(\phi_r^s(G(n,p))=0\bigr)
 \sim(s+1)\binom{n-1}{r-1}p^{\binom r2}.
\]
The one-root reduction suggests studying the cheapest local ways in
which the factor-count process can stop.  One possible approach is to
compare these witnesses with the event that the root lies in no copy
of $K_r$.  The exponential-moment and Janson estimates used here keep
only a fixed positive constant in their exponents, so they do not
provide the precision needed for this comparison.

\subsection*{A linear number of uncovered vertices}

Theorem~\ref{thm:main} is uniform for every $s=o(n)$.  The proof uses
this assumption in a transparent place: every graph in the adaptive
extraction has $n-o(n)$ vertices, so all clique scales are uniformly
comparable with their values on $n$ vertices.  If $s$ is a positive
fraction of $n$, the remaining order changes by a constant factor during
the iteration.  The one-root estimates should then be written in terms
of the current number of vertices rather than the original value of
$n$, and the local obstruction must be integrated along the whole
chain.  It would be interesting to determine the corresponding rate
function for $s=\Theta(n)$.

\subsection*{General graph factors}

The proof is also specialised to cliques.  The sharp-threshold theory for
strictly $1$-balanced factors~\cite{BHKMP24} suggests the broader
question of whether an analogous rare-event estimate holds for a fixed
strictly $1$-balanced graph $F$.  The main difficulty is structural: the
median replacement and the root-star polynomials used here exploit the
simple extension structure of a clique.  For a general $F$, one would
need a replacement which again turns every stopping pattern into a
conditional event supported on a bounded number of root coordinates,
and preferably on one root.

\medskip

\section*{Acknowledgement}
I am grateful to my advisor, Professor Rob Morris, for leading me to
this question and to combinatorics, and for his helpful guidance and
discussions during my PhD studies.  I also thank Professor Hong Liu for
helpful discussions on this project.

\medskip

\noindent\textbf{Declaration on the use of generative AI.}
During the preparation of this manuscript, the author used GPT-5.5 Pro
and GPT-5.6 Pro to explore an Efron--Stein/Slutsky approach to the
centred collision-loss estimate in
Theorem~\ref{thm:matching-clt-intro}. The main contribution of AI is refining the factor-tail argument, including using the common birth-time coupling to substitute the transfer between $G(n,p)$ and $G(n,M)$, and removing an $\varepsilon$-gap on the $p$-range from an early version:
$$n^{-2/r}(\log n)^{1/\binom r2}\ll p\leq n^{-2/(r+1)-\varepsilon}.$$
To do this, GPT 5.6 Pro provides the careful estimate in Lemma~\ref{lem:coefficient-average-bounds} (with proof presented in Appendix ~\ref{app:coefficient-profile}) and some related adjustments. The author also used GPT to check the proof details and modify the presentation.

The main idea and the proof framework of developing the adjusted JKV method and application to chromatic number belongs to the author (discussion with Prof Rob Morris). Every argument, calculation and citation generated by AI was independently checked, revised and incorporated into the manuscript by
the author, who wrote the final proofs and takes full responsibility for the paper.

\appendix

\section{Proof of the averaged rooted-coefficient estimate}
\label{app:coefficient-profile}

We prove Lemma~\ref{lem:coefficient-average-bounds} by first estimating
one coefficient.  Put $d=r-2$, $e=\binom r2$,
$\eta_r(q)=n^dq^{e-1}$ and $T(q)=1+\eta_r(q)$.
Choose $0<\sigma<1/\max\{1,d-1\}$ and set
$L_n=(\log n)^{1+\sigma}$.  For constants $c_0,K>0$, define
\[
 R(q):=K(\eta_r(q)+\log n),\qquad
 \Psi_q(z):=\min\{z,R(q)\}^2
 e^{c_0\min\{z,R(q)\}/T(q)}.
\]
We choose $\sigma$ first, then $c_0=c_0(r)$, and finally
$K=K(r,B_0)$ for the prescribed tail exponent $B_0$.

\begin{lemma}
\label{lem:truncated-coefficient}
For every fixed $B_0>0$ there are constants $K=K(r,B_0)$,
$c_0=c_0(r)>0$ and $C=C(r,B_0)>0$ such that the following holds
uniformly for $p/8\le q\le1$, $x\ne u$ and $0\le k\le d$.  If
$G_q\sim G(n,q)$ and $Y=\Gamma_{x,u,k}(G_q,q)$, then
\[
 \Pr(Y>R(q))\le n^{-B_0},
 \qquad
 \EE\Psi_q(Y)\le C\eta_r(q)T(q).
\]
The same bounds hold after deleting vertices
and after arbitrary pointwise deletion of summands or multiplication by weights in $[0,1]$, including every colour-restricted or
weighted sub-sum of~\eqref{eq:Gamma-local}.
\end{lemma}

For small $\eta_r(q)$, we group overlapping indexing sets into connected
components and derive an exponential tail up to the cap $R(q)$.  Tail
integration then gives the second assertion.  For large $\eta_r(q)$,
polynomial concentration suffices.

\begin{proof}[Proof of Lemma~\ref{lem:truncated-coefficient}]
Fix $q,x,u,k$ and abbreviate $\eta=\eta_r(q)$, $T=T(q)$ and
$\beta=(d+3)/2$, so $\beta d=e-1$.
For $B\in\binom{[n]\setminus\{x,u\}}d$ and $S\in\binom Bk$, let
$X_{B,S}$ indicate that $u\cup B$ spans a clique and every edge from
$x$ to $S$ is present.  Then
\[
 Y=q^{d-k}\sum_B\sum_{S\in\binom Bk}X_{B,S},
 \qquad
 \EE Y=\binom{n-2}{d}\binom dk q^{e-1}=\Theta_r(\eta).
\]
Set $C_d:=\max_j\binom dj$ and
$Z_B:=C_d^{-1}q^{d-k}\sum_{S\in\binom Bk}X_{B,S}$.
Thus $Y=C_d\sum_BZ_B$ and $0\le Z_B\le1$.
Disjoint indexing sets use disjoint edge coordinates, so intersection
of the $d$-sets is a dependency graph for $(Z_B)_B$.

\begin{claim}\label{claim:1}
Assume $\eta\le L_n$, and let
\[
 C_h:=\sum_{\substack{\mathcal S\subseteq
        \binom{[n]\setminus\{x,u\}}d\\
        |\mathcal S|=h,\ \mathcal S\text{ connected}}}
       \EE\prod_{B\in\mathcal S}Z_B.
\]
Then $C_1=\Theta_r(\eta)$.  If $d=1$, then $C_h=0$ for $h\ge2$.
If $d\ge2$, uniformly for $2\le h\le\lfloor R(q)\rfloor$,
\[
 C_h\le
 \begin{cases}
  \eta n^{-c_rh^{1/d}},&\eta\le1,\\
  \eta^h n^{-c_rh^{1/d}},&1\le\eta\le L_n.
 \end{cases}
\]
\end{claim}

\begin{poc}
The assertion for $h=1$ follows from the mean of $Y$.
For $d=1$, distinct indexing sets are disjoint.  Suppose $d\ge2$,
and let $B_1,\ldots,B_h$ be a connected family with union of size $v$.
Since $h\le\binom vd$, we have $v-d\ge c_dh^{1/d}$.
Order the sets so that each meets the union of its predecessors.
For $i\ge2$, let $b_i$ be the size of that intersection and put
$a_i=d-b_i$.

Fix choices $S_i\in\binom{B_i}k$.
The first block, including the factor $q^{d-k}$, has total
$q$-exponent $\beta d$.  In the $i$th block, at most
$\binom{b_i}{2}+2b_i$ required edges have appeared earlier.
Its new total exponent is at least
\[
 \frac{d(d+3)-b_i(b_i+3)}2
 =\beta a_i+\frac{a_ib_i}{2}
 \ge\beta a_i+\frac{a_i}{2}.
\]
Since $\sum_{i\ge2}a_i=v-d$, the joint contribution is at most
$q^{\beta v+(v-d)/2}$.

We next count the possible families.  The assumption $\eta\le L_n$
gives $q\le n^{-2/(d+3)}L_n^{1/(e-1)}$, so
$\log(1/q)\ge c_r\log n$.
For $h\le R\le C_{r,K}(\log n)^{1+\sigma}$ and every feasible $v$,
\[
 \frac{h\log(h+1)}{(v-d)\log(1/q)}
 \le C_{r,K}(\log n)^{(1+\sigma)(d-1)/d-1}\log\log n=o(1).
\]
Thus $C_r^h\binom{\binom vd}{h}\le q^{-(v-d)/4}$ for large $n$.
Choosing the union and then the family inside it gives
\[
 C_h\le C_r\sum_v
 n^vq^{\beta v+(v-d)/2}C_r^h\binom{\binom vd}{h}
 \le C_r\sum_v\eta^{v/d}q^{(v-d)/4},
\]
where both sums run over the feasible union sizes
$d+c_dh^{1/d}\le v\le dh$.
If $\eta\le1$, then $\eta^{v/d}\le\eta$; otherwise it is at most
$\eta^h$.  The remaining factor is at most $n^{-c_rh^{1/d}}$,
and the $O_r(h)$ possible values of $v$ are absorbed by decreasing
$c_r$.  This proves the claim.
\end{poc}

\begin{claim}\label{claim:2}
After choosing $c_0=c_0(r)>0$ sufficiently small and then
$K=K(r,B_0)$ sufficiently large, the two conclusions of the lemma hold
whenever $\eta\le L_n$.
\end{claim}
\begin{poc}
Let $a_\ell:=\EE e_\ell((Z_B)_B)$ and
$Q_M(z):=\exp(\sum_{j=1}^MC_jz^j)$.  Every family of indexing sets
splits uniquely into connected components.  Different components use
disjoint edge coordinates, so their weights multiply.  In the expansion of $Q_M$, the orders of each collection of components
cancel the factorial, giving its exact weight; all extra terms are
non-negative.  Thus, with $\preceq$ denoting coefficient-wise
domination,
\[
 \sum_{\ell=0}^Ma_\ell z^\ell\preceq Q_M(z).
\]

Claim~\ref{claim:1} gives, uniformly for $M\le R$,
$\sum_{h=2}^M h2^hC_h=o(\eta)$.  Indeed, for $\eta\le1$ its terms are
at most $\eta h2^hn^{-c_rh^{1/d}}$.  For $1\le\eta\le L_n$, the
positive term $h\log(2\eta)+\log h$ in their logarithms is
$o(h^{1/d}\log n)$, since
$R^{1-1/d}\log\log n=o(\log n)$.  When $d=1$, the sum is empty.
Since $C_1=O_r(\eta)$, it follows that
$Q'_M(2)\le C_r\eta e^{C_r\eta}$.

Fix $3C_d\le t\le R$ and set $M:=\lfloor t/(3C_d)\rfloor$.
If $Y>t$, then $\sum_BZ_B>3M$.  As $0\le Z_B\le1$, we can greedily
form $M$ disjoint groups of indices, each with total weight in $[1,2]$.
Selecting one index from each group shows that $e_M((Z_B)_B)\ge1$.
Consequently, coefficient-wise differentiation gives
\[
 \Pr(Y>t)\le a_M
 \le \frac{Q'_M(2)}{M2^{M-1}}
 \le C_r\eta\exp(C_r\eta-c_rt).
\]
Taking $t=R=K(\eta+\log n)$ proves the required tail bound when $K$
is large enough.

For the second assertion, put $Z:=\min\{Y,R\}$ and
$f(t):=t^2e^{c_0t/T}$.  Then
$\EE f(Z)=\int_0^R f'(t)\Pr(Y>t)\,dt$.
Choose a fixed $A=A(r)$ large enough.  For $t\le AT$, Markov's
inequality gives $\Pr(Y>t)\le C_r\eta/t$, while
$f'(t)\le C_rt$, so this part of the integral is $O_r(\eta T)$.
For $t>AT$, the preceding tail bound is at most
$C_r\eta e^{-c_rt/2}$.  Choose $c_0<c_r/4$.  Since $T\ge1$,
$f'(t)\le(2t+c_0t^2)e^{c_0t}$, and the remaining integral is
$O_r(\eta)$.  Thus $\EE\Psi_q(Y)=O_r(\eta T)$, as required.
\end{poc}

\begin{claim}\label{claim:3}
After increasing $K=K(r,B_0)$ if necessary, the two conclusions of the
lemma hold whenever $\eta>L_n$.
\end{claim}
\begin{poc}
If $q=1$, then $G_q$ is complete and
$Y=\Theta_r(n^d)=\Theta_r(\eta)$ deterministically.  After increasing
$K$ and $C$, both the tail bound and the estimate
$\EE\Psi_q(Y)=O_r(\eta T)$ are immediate.  We may therefore assume
$q<1$.

The coefficient case of Lemma~\ref{lem:rooted-derivatives} gives
$\max_{L\ne\varnothing}\EE'_L[Y]\le n^{-\alpha}\EE Y$ for some
$\alpha=\alpha(r)>0$.  The polynomial $Y$ has bounded degree and
coefficients, uses at most $\binom n2$ edge coordinates, and has mean
$\Theta_r(\eta)\gg\log n$.  Theorem~\ref{thm:polynomial-concentration}
therefore gives, uniformly in $q,x,u,k$,
\begin{equation}
 \Pr(Y>C_r\eta)=n^{-\omega(1)}.
 \label{eq:high-eta-concentration}
\end{equation}
Choose $K$ larger than the constant in this display.  Since
$R=K(\eta+\log n)\ge K\eta$, the required tail bound follows.

On $Y\le C_r\eta$, one has $T=\Theta(\eta)$, the truncation is inactive,
and
 $Y^2e^{c_0Y/T}\le C_r\eta Y.$
Its expectation is therefore $O_r(\eta^2)=O_r(\eta T)$.  On the
exceptional event, $R=O_{r,B_0}(\eta)$ and
$\Psi_q(Y)\le C_{r,B_0}n^{2d}$, so the uniform
$n^{-\omega(1)}$ estimate in~\eqref{eq:high-eta-concentration} makes
this contribution $o(\eta T)$.  This proves the claim.
\end{poc}

Claims~\ref{claim:2} and~\ref{claim:3} cover both regimes.  Every
restricted or weighted sum in the statement is pointwise at most $Y$.
The two conclusions therefore pass to it by monotonicity.
\end{proof}

\begin{proof}[Proof of Lemma~\ref{lem:coefficient-average-bounds}]
Take $B_0=20$ in Lemma~\ref{lem:truncated-coefficient}, and use the
common grid $\mathfrak Q=\{q_0,\ldots,q_J\}$.
Write $\eta_i=\eta_r(q_i)$, $T_i=T(q_i)$ and $A_i=A(q_i)$, and set
\[
 \mathsf E_i:=\sum_{xu\in E(G_{q_i})}\Psi_{q_i}(D_{xu}(G_{q_i})),
 \qquad
 \mathsf C_i:=\sum_{x\ne u}\sum_{k=0}^d
 \Psi_{q_i}(\Gamma_{x,u,k}(G_{q_i},q_i)).
\]
The coordinate $xu$ is independent of $D_{xu}=\Gamma_{x,u,d}$,
since this coefficient counts copies of $K_r-xu$.
The one-coefficient bound therefore gives
$\EE\mathsf E_i\le C_rn^2q_i\eta_iT_i$ and
$\EE\mathsf C_i\le C_rn^2\eta_iT_i$.
Since $\sum_iA_i^{-1}=O_r(1)$ on the geometric grid, the variable
\[
 \Xi:=\sum_{i=0}^JA_i^{-1}
 \left(\frac{\mathsf E_i}{n^2q_i\eta_iT_i}
       +\frac{\mathsf C_i}{n^2\eta_iT_i}\right)
\]
has expectation at most a constant $C_*=C_*(r)$.
The decreasing event $\cU_{\rm av}:=\{\Xi\le4C_*\}$ has probability
at least $3/4$ by Markov's inequality.  On this event,
$\mathsf E_i\le C_rA_in^2q_i\eta_iT_i$ and
$\mathsf C_i\le C_rA_in^2\eta_iT_i$ for every $i$.

There are $O_r(n^2\log n)$ choices of $(i,x,u,k)$.
The tail bound gives a decreasing event $\cU_{\rm max}$ of
probability $1-o(1)$ on which every coefficient is at most $R(q_i)$.
This includes every edge load, since $D_{xu}=\Gamma_{x,u,d}$.
On $\cU_{\rm max}$ the truncations in both sums are inactive.
Thus $\cU_{\rm coeff}:=\cU_{\rm av}\cap\cU_{\rm max}$ has
probability at least $2/3$ and gives all three estimates at every
grid point.

For $q\in[p/8,1]$, take the first grid point $q_i\ge q$; then
$q_i\le2q$ and $G_q\subseteq G_{q_i}$.
Lemma~\ref{lem:fixed-rank-monotonicity}, with a smaller exponential
constant, transfers the estimates to $q$.  Induced subgraphs and
arbitrary pointwise deletion or $[0,1]$-weighting of terms only reduce
the coefficients and the sums, so the same event covers every
restriction in the statement.
\end{proof}

\section{Rooted extension calculations}
\label{app:rooted-derivatives}

We collect the derivative estimates for the three families of graph
polynomials used above.  This also verifies all applications of
Theorem~\ref{thm:polynomial-concentration}.

\begin{lemma}
\label{lem:rooted-derivatives}
There is $c_*=c_*(r)>0$ with the following properties, uniformly for
$p/8\le q<1$.  For a rooted clique-degree polynomial from
Lemma~\ref{lem:rooted-degree-caps}, or a coefficient
$Y=\Gamma_{x,u,k}(G_q,q)$ with $\eta_r(q)>L_n$, every non-constant
non-empty derivative has expectation at most $n^{-c_*}$ times the
mean.  For a dominating $t$-rooted extension polynomial from
Lemma~\ref{lem:janson-overlap-caps}, with mean $\mu_{\ell,t}$,
\[
 \max_{L\ne\varnothing}\EE'_L[f]
 \le
 \begin{cases}
  n^{-c_*}\mu_{\ell,t},&\mu_{\ell,t}\ge(\log n)^2,\\
  n^{-c_*},&\mu_{\ell,t}<(\log n)^2.
 \end{cases}
\]
\end{lemma}

\begin{proof}
The same counting rule applies to each family.  Suppose a monomial
has $v$ free vertices, $m$ required edges, and a deterministic factor
$q^b$.  A derivative supported on $a$ free vertices fixes those
vertices and at most $m_a$ edge coordinates.  If $D_a$ is its expected
non-constant part, then
\[
 D_a=O_r(n^{v-a}q^{b+m-m_a}),
 \qquad
 \frac{D_a}{\EE f}=O_r(n^{-a}q^{-m_a}).
\]
Here and below, only feasible derivatives with a non-zero
non-constant part need be considered.

For copies of $K_\ell$ through one root, take $v=\ell-1$,
$m=\binom\ell2$ and $b=0$.  If $a<v$, then
$m_a\le\binom{a+1}{2}$; if $a=v$, at least one coordinate remains,
so $m_a\le m-1$.  Since $q\ge n^{-2/r}$ for large $n$, the derivative
to mean ratios are bounded, respectively, by constant multiples of
$n^{-a(r-a-1)/r}$ and $n^{-2/r}$.  Their exponents are uniformly
negative for $2\le\ell\le r$.

For $Y=\Gamma_{x,u,k}$, take $v=d=r-2$,
$m=\binom{d+1}{2}+k$ and $b=d-k$.  The coordinate $xu$ is absent.
Thus $m_a\le\binom a2+2a$ when $a<d$, whereas $m_d\le m-1\le e-2$.
The assumption $n^dq^{e-1}>L_n$ gives $q\ge n^{-2/(d+3)}$.
The corresponding ratios are therefore at most constant multiples
of $n^{-a(d-a)/(d+3)}$ and $n^{-2/(d+3)}$.  This proves the
coefficient assertion.

Finally, fix an overlap pattern with $t$ roots and put
$v=\ell-1-t$ and $m=tv+\binom v2$.  Edges inside the root set are
not required.  The mean is $\mu=\Theta_r(n^vq^m)$, and we may take
$m_a=ta+\binom a2$ for $a<v$ and $m_v=m-1$.
In either case $m_a/a<m/v$, so $c_a:=a-vm_a/m>0$.
Substituting $q^m=\Theta_r(\mu/n^v)$ in the counting rule gives
\[
 \frac{D_a}{\mu}\le C_rn^{-c_a}\mu^{-m_a/m},
 \qquad
 D_a\le C_rn^{-c_a}\mu^{1-m_a/m}.
\]
For $\mu\ge\log^2n$, use the first bound; for $\mu<\log^2n$, use
the second.  Absorbing logarithmic factors and taking the minimum
over the finitely many patterns proves the lemma.

All these polynomials have bounded degree and coefficients and use at
most $\binom n2$ Bernoulli coordinates.  The rooted degree means are
$\omega(\log n)$ by the lower bound on $p$, and the coefficient means
are $\Theta_r(\eta_r(q))\gg\log n$ in the range used here.  Thus the
first part of Theorem~\ref{thm:polynomial-concentration} applies to
these families and to overlap patterns with mean at least $\log^2n$.
For the remaining overlap patterns, take $M=C\log^2n$ and choose the
derivative exponent smaller than $c_*$; then the second part applies.
The endpoint $q=1$ is deterministic.
\end{proof}

\end{document}